\documentclass[titlepage,12pt]{article} 
\usepackage{hyperref}
\usepackage[usenames,dvipsnames]{pstricks} 
\usepackage{pst-plot} 
\usepackage[title]{appendix}
\usepackage{amssymb,amsthm,amsmath} 
\usepackage[a4paper]{geometry}
\usepackage{datetime2}
\usepackage[utf8]{inputenc}
\usepackage[italian,english]{babel}

\date{}

\newcommand{\ep}{\varepsilon}
\newcommand{\re}{\mathbb{R}}

\newcommand{\n}{\mathbb{N}}

\newcommand{\holder}{H\"older}
\newcommand{\PS}{\mathcal{PS}}
\newcommand{\ek}{E_{\mathrm{Kov}}}
\newcommand{\eh}{E_{\mathrm{Hyp}}}
\newcommand{\et}{E_{\mathrm{Tar}}}
\newcommand{\al}{a_{\lambda}}
\newcommand{\bl}{b_{\lambda}}
\newcommand{\cl}{c_{\lambda}}
\newcommand{\vl}{u_{\lambda}}
\newcommand{\oml}{\omega_{\lambda}}
\newcommand{\psil}{\psi_{\lambda}}
\newcommand{\epl}{\ep_{\lambda}}

\newcommand{\chat}{c_{*}}
\newcommand{\cd}{c_{\delta}}
\newcommand{\PSu}{\mathcal{PS}^{(1)}}
\newcommand{\PSd}{\mathcal{PS}^{(2)}}

\newtheorem{thm}{Theorem}[section]
\newtheorem{thmbibl}{Theorem}

\newtheorem{rmk}[thm]{Remark}
\newtheorem{prop}[thm]{Proposition}
\newtheorem{defn}[thm]{Definition}

\newtheorem{lemma}[thm]{Lemma}

\title{Finite vs infinite derivative loss for abstract wave equations with singular time-dependent propagation speed}

\author{Marina Ghisi\vspace{1ex}\\ 
{\normalsize Università degli Studi di Pisa} \\
{\normalsize Dipartimento di Matematica}\\ 
{\normalsize PISA (Italy)}\\
{\normalsize e-mail: \texttt{marina.ghisi@unipi.it}}
\and
Massimo Gobbino\vspace{1ex}\\ 
{\normalsize Università degli Studi di Pisa} \\
{\normalsize Dipartimento di Ingegneria Civile e Industriale}\\ 
{\normalsize PISA (Italy)}\\  
{\normalsize e-mail: \texttt{massimo.gobbino@unipi.it}}
}

\begin{document}
\maketitle

\begin{abstract}

We consider an abstract wave equation with a propagation speed that depends only on time. We investigate well-posedness results with finite derivative loss in the case where the propagation speed is smooth for positive times, but potentially singular at the initial time. 

We prove that solutions exhibit a finite derivative loss under a family of conditions that involve the blow up rate of the first and second derivative of the propagation speed, in the spirit that the weaker is the requirement on the first derivative, the stronger is the requirement on the second derivative. Our family of conditions interpolates between the two limit cases that were already known in the literature.

We also provide the counterexamples that show that, as soon as our conditions fail, solutions can exhibit an infinite derivative loss. The existence of such pathologies was an open problem even in the two extreme cases.

\vspace{6ex}

\noindent{\bf Mathematics Subject Classification 2010 (MSC2010):} 
35L90 (35L20, 35B30, 35B65).

		
\vspace{6ex}

\noindent{\bf Key words:} 
linear hyperbolic equation, wave equation, propagation speed, finite derivative loss, infinite derivative loss, Baire category, residual set.

\end{abstract}

 
\section{Introduction}

Let $H$ be a real Hilbert space, and let $A$ be a linear nonnegative self-adjoint operator on $H$. We consider, in a time interval $[0,T_{0}]$, the evolution equation
\begin{equation}
u''(t)+c(t)Au(t)=0,
\label{eqn:main}
\end{equation}
with initial data
\begin{equation}
u(0)=u_{0},
\qquad
u'(0)=u_{1}.
\label{eqn:main-data}
\end{equation}

This is an abstract model, for example, of the wave equation
\begin{equation}
u_{tt}-c(t)\Delta u=0
\nonumber
\end{equation}
with suitable boundary conditions in an open set $\Omega\subseteq\re^{d}$. For consistency with this model case, the function $c(t)$ is usually called the ``propagation speed''.

\paragraph{\textmd{\textit{Regularity of the propagation speed vs regularity of solutions}}}

It is well-known that problem (\ref{eqn:main})--(\ref{eqn:main-data}) admits a unique solution for large classes of initial data, even if $c(t)$ is just in $L^{1}((0,T_{0}))$. This solution, however, is very weak and lives in a huge space of hyperdistributions. Let us recall some classical notions concerning the regularity of solutions.

\begin{defn}[Well-posedness vs derivative loss]\label{defn:regularity}
\begin{em}
\mbox{} 

\begin{itemize}

\item  (Well-posedness in Sobolev spaces). Problem (\ref{eqn:main})--(\ref{eqn:main-data}) is said to be \emph{well-posed in Sobolev spaces} if, for every real number $\beta$, and for every pair of initial data $(u_{0},u_{1})\in D(A^{\beta+1/2})\times D(A^{\beta})$, the unique solution belongs to the space
\begin{equation}
C^{0}\left([0,T_{0}],D(A^{\beta+1/2})\right)\cap C^{1}\left([0,T_{0}],D(A^{\beta})\right),
\nonumber
\end{equation}
and there exists a constant $M$, that depends on the propagation speed but not on the solution, such that
\begin{equation}
\|A^{\beta}u'(t)\|^{2}+\|A^{\beta+1/2}u(t)\|^{2}\leq
M\left(\|A^{\beta}u_{1}\|^{2}+\|A^{\beta+1/2}u_{0}\|^{2}\right)
\qquad
\forall t\in[0,T_{0}].
\nonumber
\end{equation}

\item  (Finite derivative loss). Problem (\ref{eqn:main})--(\ref{eqn:main-data}) is said to be well-posed with \emph{finite derivative loss} if there exist positive real numbers $\delta$ and $M$ such that, for every real number $\beta$, and for every pair of initial data $(u_{0},u_{1})\in D(A^{\beta+1/2})\times D(A^{\beta})$, the unique solution satisfies
\begin{equation}
u\in C^{0}\left([0,T_{0}],D(A^{\beta-\delta+1/2})\right)\cap C^{1}\left([0,T_{0}],D(A^{\beta-\delta})\right),
\nonumber
\end{equation}
and 
\begin{equation}
\|A^{\beta-\delta}u'(t)\|^{2}+\|A^{\beta-\delta+1/2}u(t)\|^{2}\leq 
M\left(\|A^{\beta}u_{1}\|^{2}+\|A^{\beta+1/2}u_{0}\|^{2}\right)
\qquad
\forall t\in[0,T_{0}].
\nonumber
\end{equation}

\item (Infinite derivative loss).  Problem (\ref{eqn:main})--(\ref{eqn:main-data}) is said to exhibit an \emph{infinite derivative loss} if there exists a solution $u(t)$ such that 
\begin{equation}
(u(0),u'(0))\in D(A^{\beta+1/2})\times D(A^{\beta})
\qquad
\forall\beta>0,
\nonumber
\end{equation}
but
\begin{equation}
(u(t),u'(t))\not\in D(A^{-\beta+1/2})\times D(A^{-\beta})
\qquad
\forall\beta>0,
\quad
\forall t\in(0,T_{0}].
\nonumber
\end{equation}

\end{itemize}

\end{em}
\end{defn}

We observe that finite derivative loss implies in particular that, when initial data satisfy $(u_{0},u_{1})\in D(A^{\infty})\times D(A^{\infty})$, the solution satisfies $(u(t),u'(t))\in D(A^{\infty})\times D(A^{\infty})$ for every $t\in[0,T_{0}]$, where $D(A^{\infty})$ is defined as the intersection of $D(A^{\beta})$ as $\beta$ ranges over all positive real numbers. In the concrete case of the wave equation this means that initial data of class $C^{\infty}$ give rise to solutions of class $C^{\infty}$. 

On the contrary, infinite derivative loss means that there exists at least one solution whose initial data are in $D(A^{\infty})$, but that for all subsequent times is not even a distribution. We observe that, when such a pathological solution exists, then all solutions that are not orthogonal to it have the same pathological behavior, due to the linearity of the equation. We observe also that, even in the case of solutions with an infinite derivative loss, it might happen that either $u(t)$ or $u'(t)$ are more regular at some special positive times, but they can not be both regular at the same time.

In the case of (\ref{eqn:main}), it is well-known that the regularity of solutions depends on the sign and on the regularity of the propagation speed. Concerning the sign, in the sequel we always assume that the propagation speed satisfies the strict hyperbolicity condition 
\begin{equation}
0<\mu_{1}\leq c(t)\leq\mu_{2}
\qquad
\forall t\in(0,T_{0})
\label{hp:c-sh}
\end{equation}
for suitable positive real numbers $\mu_{1}$ and $\mu_{2}$. Under this assumption, the classical result is that higher ``space'' regularity of initial data compensates lower time-regularity of the propagation speed, as described below.

\begin{itemize}

\item If $c(t)$ is Lipschitz continuous, or more generally with bounded variation, then problem (\ref{eqn:main})--(\ref{eqn:main-data}) is well-posed in Sobolev spaces. 

\item If $c(t)$ is log-Lipschitz continuous, then problem (\ref{eqn:main})--(\ref{eqn:main-data}) is well-posed with finite derivative loss. 

\item  If $c(t)$ is $\alpha$-\holder\ continuous for some $\alpha\in(0,1)$, then problem (\ref{eqn:main})--(\ref{eqn:main-data}) can exhibit an infinite derivative loss. More precisely, the problem is globally well-posed in Gevrey spaces of order $s<(1-\alpha)^{-1}$, locally well-posed in Gevrey spaces of order $s=(1-\alpha)^{-1}$, and pathologies are possible with initial data in Gevrey spaces of order $s>(1-\alpha)^{-1}$, and a fortiori in Sobolev spaces. 

\end{itemize}

These features were observed for the first time by F.~Colombini, E.~De Giorgi and S.~Spagnolo in the seminal paper~\cite{dgcs}, and then widely extended by considering more general continuity moduli or propagation speeds depending also on the space variables (see for example~\cite{2006-JDE-CicCol,2015-CPDE-ColDSaFanMet,1995-Duke-ColLer}). We refer to~\cite[section~2]{gg:dgcs-strong} for a review, in the abstract setting, of the relations between the continuity modulus of the propagation speed and the exact spaces where the problem is well-posed.

\paragraph{\textmd{\textit{Propagation speeds that are singular at the origin}}}

In a different direction, one might be interested in understanding how a singularity of the propagation speed at a single time affects the regularity of solutions. This research direction is motivated also by potential applications to the study of quasilinear problems in a neighborhood of the time when a smooth solution has blow up (see for example~\cite[section~4]{2003-MMAS-Hirosawa}). As a model case, one can consider a propagation speeds $c(t)$ that is of class $C^{1}$ in $(0,T_{0}]$, but whose derivative satisfies
\begin{equation}
|c'(t)|\leq\frac{K}{t^{\theta}}
\qquad
\forall t\in(0,T_{0}]
\label{hp:c'-theta}
\end{equation}
for some positive real numbers $\theta$ and $K$. As always, we assume that the strict hyperbolicity condition (\ref{hp:c-sh}) holds true, and we can also assume that $c(t)$ is continuous up to $t=0$, since pathologies originate from the potential accumulation of oscillations of the propagation speed in the origin. 

Model cases of this type have been studied in the last two decades. If $\theta<1$ there is almost nothing to say, because in this case the propagation speed has finite total variation, and hence we already know from the previous theory that the problem is well-posed in Sobolev spaces. The case with $\theta>1$ was addressed by F.~Colombini, D.~Del Santo and T.~Kinoshita in~\cite{2002-SNS-ColDSaKin}. The result is that the problem is globally well-posed in Gevrey spaces of order $s<\theta/(\theta-1)$, locally well-posed in Gevrey spaces of order $s=\theta/(\theta-1)$, while an infinite derivative loss is possible when initial data are in Gevrey spaces of order $s>\theta/(\theta-1)$.

The borderline case $\theta=1$ is more challenging. A first result in this case had already been obtained in 1990 by T.~Yamazaki~\cite{1990-CPDE-yamazaki}, who proved that (\ref{hp:c'-theta}), together with a bound on the second derivative of the form $|c''(t)|\leq K/t^{2}$, yields well-posedness in Sobolev spaces. More recently, in~\cite{2002-SNS-ColDSaKin} it was proved that the sole assumption on the first derivative is enough to guarantee at least a finite derivative loss.

\begin{thmbibl}[{see~\cite[Theorem~1]{2002-SNS-ColDSaKin}}]\label{thmbibl:c'}

Let $H$ be a real Hilbert space, and let $A$ be a linear nonnegative self-adjoint operator on $H$. 

Let $T_{0}>0$ be a real number, and let $c:(0,T_{0}]\to\re$ be a function of class $C^{1}$ satisfying the strict hyperbolicity condition (\ref{hp:c-sh}) and the estimate 
\begin{equation}
|c'(t)|\leq\frac{K}{t}
\qquad
\forall t\in(0,T_{0}]
\label{hp:c'-1}
\end{equation}
for a suitable real number $K$.

Then problem (\ref{eqn:main})--(\ref{eqn:main-data}) exhibits a finite derivative loss.

\end{thmbibl}

The optimality of Theorem~\ref{thmbibl:c'} was not addressed in~\cite{2002-SNS-ColDSaKin}. On the contrary, in the same years many authors obtained again well-posedness results with finite derivative loss under assumptions on the first derivative weaker than (\ref{hp:c'-1}), provided that further conditions are added on the modulus of continuity of $c(t)$ (see for example~\cite{2007-DIE-DSaKinRei}) or on the second derivative. The prototype of these results is the following (compare also with~\cite{2003-MMAS-Hirosawa,2003-MatN-Hirosawa,2003-CPDE-KubRei}).

\begin{thmbibl}[{see~\cite[Theorem~1.1]{2003-BSM-ColDsaRei}}]\label{thmbibl:c''}

Let $H$ be a real Hilbert space, and let $A$ be a linear nonnegative self-adjoint operator on $H$. 

Let $T_{0}\in(0,1)$ be a real number, and let $c:(0,T_{0}]\to\re$ be a function of class $C^{2}$ satisfying the strict hyperbolicity condition (\ref{hp:c-sh}) and the two estimates
\begin{equation}
|c'(t)|\leq K\frac{|\log t|}{t}
\qquad
\forall t\in(0,T_{0}]
\label{hp:c'-log}
\end{equation}
and
\begin{equation}
|c''(t)|\leq K\frac{|\log t|^{2}}{t^{2}}
\qquad
\forall t\in(0,T_{0}]
\label{hp:c''}
\end{equation}
for a suitable positive real number $K$.

Then problem (\ref{eqn:main})--(\ref{eqn:main-data}) exhibits a finite derivative loss.

\end{thmbibl}

A partial step toward optimality of these results was obtained again in~\cite{2003-BSM-ColDsaRei}.

\begin{thmbibl}[{see~\cite[Theorems~1.2 and~1.3]{2003-BSM-ColDsaRei}}]\label{thmbibl:counter}

Let $\omega:(0,T_{0}]\to(0,+\infty)$ be a nonincreasing function such that $\omega(t)\to +\infty$ as $t\to 0^{+}$.

Then there exists a propagation speed $c:(0,T_{0}]\to\re$ of class $C^{\infty}$, continuous up to $t=0$, that satisfies the strict hyperbolicity condition (\ref{hp:c-sh}) and the two estimates
\begin{equation}
|c'(t)|\leq \frac{|\log t|}{t}\omega(t)
\qquad
\forall t\in(0,T_{0}]
\nonumber
\end{equation}
and
\begin{equation}
|c''(t)|\leq \frac{|\log t|^{2}}{t^{2}}\omega(t)
\qquad
\forall t\in(0,T_{0}],
\nonumber
\end{equation}
and for which problem (\ref{eqn:main})--(\ref{eqn:main-data}) exhibits an infinite derivative loss.

\end{thmbibl}

At least three questions remained unsolved.

\begin{enumerate}

\renewcommand{\labelenumi}{(Q\arabic{enumi})}

\item  Considering only assumptions on first derivatives, determine whether (\ref{hp:c'-log}), or any  condition stronger than (\ref{hp:c'-1}) but weaker than (\ref{hp:c'-log}), is enough to guarantee alone a finite derivative loss.

\item  If the answer to the previous question is negative, determine whether a bound on the second derivative weaker than (\ref{hp:c''}), in addition to (\ref{hp:c'-log}), is enough to guarantee a finite derivative loss.

\item  The positive results of Theorem~\ref{thmbibl:c'} and Theorem~\ref{thmbibl:c''} seem to be somewhat isolated and independent, in the sense that the conclusion is the same, but the assumptions are different and none of them implies the other. In these cases it reasonable to expect the existence of a more general result that in some sense ``connects them''.
 
\end{enumerate}

In this paper we address the three questions above, and in particular we fill the gap between the positive results of Theorems~\ref{thmbibl:c'} and~\ref{thmbibl:c''}, and the counterexamples of  Theorem~\ref{thmbibl:counter}.

\paragraph{\textmd{\textit{Our contribution}}}

The first result of this paper is a sort of bridge between Theorem~\ref{thmbibl:c'} and Theorem~\ref{thmbibl:c''}. Indeed, in Theorem~\ref{thm:main-fdl} we consider propagation speeds whose first time-derivative grows more than in (\ref{hp:c'-1}) but less than in (\ref{hp:c'-log}), and again we prove well-posedness with finite derivative loss provided that the second derivative of the propagation speed satisfies a suitable estimate. Roughly speaking, the more we ask on the first derivative, the less we need to ask on the second derivative. This result answers question (Q3).

The second contribution of this paper is a family of counterexamples. In particular, in Theorem~\ref{thm:counter-c'} we show the optimality of Theorem~\ref{thmbibl:c'}, while in Theorem~\ref{thm:counter-c''} we show the optimality of Theorem~\ref{thmbibl:c''} and of our  Theorem~\ref{thm:main-fdl}. In other words, the answer to questions (Q1) and (Q2) is negative, and the previous gap between positive and negative results is now entirely filled by counterexamples.

\paragraph{\textmd{\textit{Overview of the technique}}}

From the technical point of view, the spectral theorem reduces (\ref{eqn:main}) to the family of ordinary differential equations
\begin{equation}
u_{\lambda}''(t)+\lambda^{2}c(t)u_{\lambda}(t)=0,
\label{eqn:u-lambda}
\end{equation}
where $\lambda$ is a positive real parameter. Proving that (\ref{eqn:main}) is well-posed in some class of initial data is equivalent to estimating the growth of solutions to (\ref{eqn:u-lambda}), namely to finding an estimate of the form 
\begin{equation}
|u_{\lambda}'(t)|^{2}+\lambda^{2}|u_{\lambda}(t)|^{2}\leq
\left(|u_{\lambda}(0)|^{2}+\lambda^{2}|u_{\lambda}(0)|^{2}\right)\cdot\phi(\lambda)
\qquad
\forall t\in[0,T_{0}]
\label{est:comp-leq}
\end{equation}
for a suitable function $\phi(\lambda)$. The growth of $\phi(\lambda)$ as $\lambda\to+\infty$ determines the spaces where (\ref{eqn:main}) is well-posed. For example, if $\phi(\lambda)$ is bounded, then (\ref{eqn:main}) is well-posed in Sobolev spaces, while if $\phi(\lambda)$ has an exponential growth of the form $\exp(M\lambda^{1/s})$ for some positive real numbers $M$ and $s$, then (\ref{eqn:main}) is well-posed in Gevrey spaces of order~$s$. 

In this paper we are interested in the case where $\phi(\lambda)$ grows as a power of $\lambda$, because this implies that (\ref{eqn:main}) is well-posed with finite derivative loss. More precisely, in Proposition~\ref{prop:main-energy} we show that, under the assumptions of our Theorem~\ref{thm:main-fdl}, there exist three positive real numbers $\delta$, $M$, $\lambda_{0}$ such that solutions to (\ref{eqn:u-lambda}) satisfy
\begin{equation}
|u_{\lambda}'(t)|^{2}+\lambda^{2}|u_{\lambda}(t)|^{2}\leq
M\left(|u_{\lambda}(0)|^{2}+\lambda^{2}|u_{\lambda}(0)|^{2}\right)\cdot\lambda^{\delta}
\qquad
\forall t\in[0,T_{0}],\quad\forall\lambda\geq\lambda_{0}.
\nonumber
\end{equation}

In the case of Theorem~\ref{thmbibl:c'}, this estimate can be obtained by exploiting an estimate for the Kovaleskyan energy
\begin{equation}
\ek(t):=|u_{\lambda}'(t)|^{2}+\lambda^{2}|u_{\lambda}(t)|^{2}
\label{defn:ek}
\end{equation}
in some interval $[0,\al]$, followed by an estimate for the hyperbolic energy
\begin{equation}
\eh(t):=|u_{\lambda}'(t)|^{2}+\lambda^{2}c(t)|u_{\lambda}(t)|^{2}
\label{defn:eh}
\end{equation}
in the interval $[\al,T_{0}]$, where $\al$ is chosen in a suitable ($\lambda$-dependent) way.

In the case of Theorem~\ref{thmbibl:c''}, estimate (\ref{est:comp-leq}) can be obtained by exploiting again the Kovaleskyan energy (\ref{defn:ek}) in some interval $[0,\bl]$, followed by an estimate of the energy
\begin{equation}
\et(t):=\frac{1}{\sqrt{c(t)}}\left(u_{\lambda}'(t)+\frac{c'(t)}{4c(t)}u_{\lambda}(t)\right)^{2}+\lambda^{2}\sqrt{c(t)}\,|u_{\lambda}(t)|^{2}
\label{defn:et-sqrt}
\end{equation}
(introduced by S.~Tarama in \cite{2007-EJDE-tarama}) in the interval $[\bl,T_{0}]$, where again $\bl$ has to be chosen in a suitable way. The same technique delivers a proof of Yamazaki's well-posedness result in Sobolev spaces.

In the case of our Theorem~\ref{thm:main-fdl} we have to mix the two strategies, and indeed we prove the key estimate (\ref{est:comp-leq}) by estimating the Kovaleskyan energy (\ref{defn:ek}) in some interval $[0,\al]$, then the hyperbolic energy (\ref{defn:eh}) in a subsequent interval $[\al,\bl]$, and finally the Tarama energy (\ref{defn:et-sqrt}) in the last interval $[\bl,T_{0}]$. Again we have to choose $0<\al<\bl<T_{0}$ in a suitable ($\lambda$-dependent) way. 

Concerning the construction of counterexamples, the infinite derivative loss follows from the existence of a propagation speed that realizes some sort of non-polynomial growth on a sequence of solutions to (\ref{eqn:u-lambda}). More precisely, we need a propagation speed $c(t)$, a sequence of real numbers $\lambda_{k}\to +\infty$, and a function $\phi(\lambda_{k})$ that grows faster than any power of $\lambda_{k}$, such that (\ref{eqn:u-lambda}) admits nontrivial solutions that for every $t\in(0,T_{0}]$ satisfy
\begin{equation}
|u_{\lambda_{k}}'(t)|^{2}+\lambda_{k}^{2}|u_{\lambda_{k}}(t)|^{2}\geq
\left(|u_{\lambda_{k}}(0)|^{2}+\lambda_{k}^{2}|u_{\lambda_{k}}(0)|^{2}\right)\cdot\phi(\lambda_{k})
\nonumber
\end{equation}
for infinitely many indices $k$, possibly dependent on $t$. We call these special propagation speeds ``universal activators'', because they cause in the same time the growth of a sequences of solutions. The construction of universal activators requires the combination of special solutions that exhibit the required growth just for a specific value of $\lambda$, and that we call ``asymptotic activators''.

In the case of~\cite{dgcs}, asymptotic activators were produced by observing that
\begin{equation}
w_{\lambda}(t):=
\sin(\gamma\lambda t)\exp\left(2\ep\gamma\lambda t-\ep\sin(2\gamma\lambda t)\strut\right)
\nonumber
\end{equation}
grows exponentially and solves (\ref{eqn:u-lambda}) with
\begin{equation}
c(t):=\gamma^{2}-8\ep\gamma^{2}\sin(2\gamma\lambda t)-16\ep^{2}\gamma^{2}\sin^{4}(\gamma\lambda t).
\nonumber
\end{equation}

Most counterexamples quoted in this introduction, including those in Theorem~\ref{thmbibl:counter}, are based on special solutions of this kind, combined in different ways by choosing $\ep$ as a function of $\lambda$. Unfortunately, in this borderline case these special solutions are not enough, and this explains the gap between the positive results and the counterexamples exhibited so far. 

In our construction we need to modify this scheme by choosing $\ep$ to be dependent also on $t$, and singular as $t\to 0^{+}$. Roughly speaking, $\ep(t)$ has to blow up in the origin as fast as the upper bound that is imposed on $|c'(t)|$. We refer to subsection~\ref{sec:as-act-icf} for the details of our construction of asymptotic activators.

Finally, asymptotic activators can be combined in at least two different ways in order to produce a universal activator. The classical approach introduced in~\cite{dgcs} requires a clever iteration procedure with many parameters to be chosen appropriately. Here we follow a more indirect path introduced in~\cite{gg:residual,gg:dgcs-critical} that leaves the dirty job to Baire category theorem and delivers \emph{not just a single counterexample, but a residual set of them}. In other words, \emph{when the assumptions that guarantee a finite derivative loss are not satisfied, the infinite derivative loss is the common behavior of solutions.}

\paragraph{\textmd{\textit{Structure of the paper}}}

This paper is organized as follows. In section~\ref{sec:statements} we state our main results. In section~\ref{sec:well-posed} we prove our well-posedness result that interpolates Theorem~\ref{thmbibl:c'} and Theorem~\ref{thmbibl:c''}. Finally, in section~\ref{sec:counterexamples} we construct the counterexamples that show the optimality of our (and previous) results.


\setcounter{equation}{0}
\section{Statements}\label{sec:statements}

\subsection{Well-posedness with finite derivative loss}

The first contribution of this paper is the following result, where we clarify the exact balance between the bounds on first and second derivatives of the propagation speed that yields a finite derivative loss.

\begin{thm}[First and second derivatives vs finite derivative loss]\label{thm:main-fdl}

Let $H$ be a real Hilbert space, and let $A$ be a linear nonnegative self-adjoint operator on $H$. 

Let $T_{0}\in(0,1)$ be a real number, and let $\omega:(0,T_{0}]\to(0,+\infty)$ and $\psi:(0,T_{0}]\to(0,+\infty)$ be two nonincreasing functions such that
\begin{equation}
\omega(t)(1+\psi(t))\leq K_{0}|\log t|
\qquad
\forall t\in(0,T_{0}]
\label{hp:omega-psi}
\end{equation}
for a suitable constant $K_{0}$.

Let $c:(0,T_{0}]\to\re$ be a function of class $C^{2}$ satisfying the strict hyperbolicity condition (\ref{hp:c-sh}) and the two estimates
\begin{equation}
|c'(t)|\leq\frac{\omega(t)}{t}
\qquad
\forall t\in(0,T_{0}]
\label{hp:fdl-c'}
\end{equation}
and
\begin{equation}
|c''(t)|\leq\frac{\omega(t)^{2}}{t^{2}}\exp(\psi(t))
\qquad
\forall t\in(0,T_{0}].
\label{hp:fdl-c''}
\end{equation}

Then problem (\ref{eqn:main})--(\ref{eqn:main-data}) exhibits a finite derivative loss in the sense of Definition~\ref{defn:regularity}.

\end{thm}

Let us comment on the assumptions of Theorem~\ref{thm:main-fdl}. First of all, we show that these assumptions range between the two extremes represented by the assumptions of Theorem~\ref{thmbibl:c'} and Theorem~\ref{thmbibl:c''}.

\begin{rmk}[From Theorem~\ref{thmbibl:c'} to Theorem~\ref{thmbibl:c''}]
\begin{em}

The statement of Theorem~\ref{thm:main-fdl} admits the possibility that $\omega(t)$ is bounded. In this limit case, however, we enter the realm of Theorem~\ref{thmbibl:c'} because assumption (\ref{hp:fdl-c'}) reduces to (\ref{hp:c'-1}), and this is enough to guarantee a finite derivative loss, without even assuming the existence of second derivatives. In other words, Theorem~\ref{thm:main-fdl} says something new only when $\omega(t)\to +\infty$ as $t\to 0^{+}$.

On the opposite side, we observe that (\ref{hp:omega-psi}) implies in particular that $\omega(t)\leq K_{0}|\log t|$, which means that the maximum blow up rate that is allowed for $c'(t)$ is the one assumed in Theorem~\ref{thmbibl:c''}. Moreover, when $\omega(t)\sim K|\log t|$, then $\psi(t)$ is necessarily bounded, and therefore (\ref{hp:fdl-c'}) and (\ref{hp:fdl-c''}) are exactly the assumptions of Theorem~\ref{thmbibl:c''}.

If we want to explore the whole spectrum, we can consider for example the family of propagation speeds
\begin{equation}
c_{\alpha}(t):=2+\exp(-|\log t|^{1-\alpha})\sin\left(|\log t|^{2\alpha}\exp(|\log t|^{1-\alpha})\right).
\nonumber
\end{equation}

When $\alpha=0$ we obtain $2+t\sin(1/t)$, which is the model case for Theorem~\ref{thmbibl:c'}. When $\alpha=1$ we obtain $2+\sin(|\log t|^{2})$, which is the model case for Theorem~\ref{thmbibl:c''}. When $\alpha\in(0,1)$ the propagation speed $c_{\alpha}(t)$ satisfies (\ref{hp:fdl-c'}) and (\ref{hp:fdl-c''}) with (up to multiplicative constants) $\omega(t):=|\log t|^{\alpha}$ and $\psi(t):=|\log t|^{1-\alpha}$, and this choice is consistent with (\ref{hp:omega-psi}).
\end{em}
\end{rmk}

\begin{rmk}[Beyond Theorem~\ref{thmbibl:c''}]
\begin{em}

In principle, one could consider extending the theory by asking even weaker assumptions on $c'(t)$ and suitable stronger assumptions on $c''(t)$. Nevertheless, this would lead to a somewhat empty case, because  from the classical Glaeser inequality we know that the first derivative of a nonnegative function is bounded by the square root of the second derivative, and hence any assumption on $c''(t)$ stronger than (\ref{hp:c''}) yields automatically an assumption on $c'(t)$ stronger than (\ref{hp:c'-log}).

\end{em}
\end{rmk}

Concerning the derivative loss, it is probably worthwhile mentioning the following subtlety. 

\begin{rmk}[Progressive vs instantaneous derivative loss]
\begin{em}

As we recalled in the introduction, problem (\ref{eqn:main})--(\ref{eqn:main-data}) exhibits a finite derivative loss also when $c(t)$ is log-Lipschitz continuous, but of course there is no relation between log-Lipschitz continuity and bounds on derivatives.  Concerning the conclusion, in the log-Lipschitz case the derivative loss is \emph{progressive}, in the sense that initial data $(u_{0},u_{1})\in D(A^{\beta+1/2})\times D(A^{\beta})$ give rise to a solution with
\begin{equation}
(u(t),u'(t))\in D(A^{\beta-\sigma t+1/2})\times D(A^{\beta-\sigma t})
\qquad
\forall t\in[0,T_{0}],
\nonumber
\end{equation}
so that the derivative loss increases with time, but tends to~0 as $t\to 0^{+}$. 

On the contrary, in our case (as well as in Theorems~\ref{thmbibl:c'} and~\ref{thmbibl:c''}) the derivative loss is \emph{instantaneous}, in the sense that the solution loses immediately a finite number of derivatives, but then remains constantly in the same space. We point out that our choice that $T_{0}\in (0,1)$ is just aimed at preventing (\ref{hp:omega-psi}) from being too restrictive. If $c(t)$ is defined and smooth enough after $T_{0}$, then the solution will remain in the same space as long as $c(t)$ is defined.  

\end{em}
\end{rmk}


The proof of Theorem~\ref{thm:main-fdl} follows in a classical way from an estimate on the growth of solutions to the family of ordinary differential equations (\ref{eqn:u-lambda}). The exact statement we need is the following.

\begin{prop}[Key estimate on the growth of components]\label{prop:main-energy}

Let $T_{0}$, $\omega$, $\psi$ and $K_{0}$ be as in Theorem~\ref{thm:main-fdl}. Let $c:(0,T_{0}]\to\re$ be a function of class $C^{2}$ satisfying the strict hyperbolicity condition (\ref{hp:c-sh}) and the estimates (\ref{hp:fdl-c'}) and (\ref{hp:fdl-c''}).

Then there exist three positive real numbers $\lambda_{0}$, $M$ and $\delta$, depending only on $T_{0}$, $\mu_{1}$, $\mu_{2}$ and $K_{0}$, such that for every $\lambda\geq\lambda_{0}$ all solutions to problem (\ref{eqn:u-lambda}) satisfy
\begin{equation}
|u_{\lambda}'(t)|^{2}+\lambda^{2}|u_{\lambda}(t)|^{2}\leq
M\left(|u_{\lambda}'(0)|^{2}+\lambda^{2}|u_{\lambda}(0)|^{2}\right)\cdot\exp(\delta\log\lambda)
\qquad
\forall t\in[0,T_{0}].
\label{th:prop-main}
\end{equation}

\end{prop}


\subsection{Infinite derivative loss}

The second contribution of this paper are some counterexamples that show that the assumptions of our Theorem~\ref{thm:main-fdl}, as well as of the previous results (Theorems~\ref{thmbibl:c'} and~\ref{thmbibl:c''}) are optimal if the operator is unbounded. For the sake of simplicity, we restrict ourselves to operators that admit an unbounded sequence of eigenvalues (this is the case, for example, of the Laplacian with different boundary conditions in bounded domains).

\begin{defn}[Unbounded positive multiplication operators]\label{defn:upmo}
\begin{em}

Let $H$ be a Hilbert space, and let $A$ be a linear operator on $H$. We say that $A$ is an \emph{unbounded positive multiplication operator} if there exist a sequence $\{e_{n}\}\subseteq H$ of orthonormal vectors (not necessarily a Hilbert basis), and an unbounded sequence $\{\lambda_{n}\}$ of positive real numbers such that
\begin{equation}
Ae_{n}=\lambda_{n}e_{n}
\qquad
\forall n\geq 1.
\nonumber
\end{equation}

\end{em}
\end{defn}

Following the ideas introduced in the recent papers~\cite{gg:residual,gg:dgcs-critical}, we show that an infinite derivative loss is the common behavior of solutions whenever the propagation speed satisfies (\ref{hp:fdl-c'}) and (\ref{hp:fdl-c''}) with growth rates $\omega$ and $\psi$ that do not satisfy (\ref{hp:omega-psi}). More precisely, we show that the set of propagation speeds for which problem (\ref{eqn:main})--(\ref{eqn:main-data}) exhibits an infinite derivative loss is residual in the spaces that we introduce below.

\begin{defn}[Special classes of propagation speeds]\label{defn:ps}
\begin{em}

Let $T_{0}>0$ and $\mu_{2}>\mu_{1}>0$ be real numbers, and let $\omega:(0,T_{0}]\to(0,+\infty)$ and $\psi:(0,T_{0}]\to(0,+\infty)$ be two nonincreasing functions.

\begin{itemize}

\item  We call $\PSu(T_{0},\mu_{1},\mu_{2},\omega)$ the set of functions $c\in C^{0}([0,T_{0}])\cap C^{1}((0,T_{0}])$ that satisfy the strict hyperbolicity condition (\ref{hp:c-sh}) and the growth estimate (\ref{hp:fdl-c'}).

\item  We call $\PSd(T_{0},\mu_{1},\mu_{2},\omega,\psi)$ the set of functions $c\in C^{0}([0,T_{0}])\cap C^{2}((0,T_{0}])$ that satisfy the strict hyperbolicity condition (\ref{hp:c-sh}) and the two growth estimates (\ref{hp:fdl-c'}) and (\ref{hp:fdl-c''}).

\end{itemize}

\end{em}
\end{defn}

\begin{rmk}[Metric space structures]\label{rmk:cms}
\begin{em}

There are several ways to define a distance in the spaces of Definition~\ref{defn:ps}. Here we exploit a weighted uniform norm, with a weight that vanishes in the origin more than the reciprocal of the growth rate of derivatives.

\begin{itemize}

\item  The set $\PSu(T_{0},\mu_{1},\mu_{2},\omega)$ is a complete metric space with respect to the distance defined by
\begin{equation}
d_{\PSu}(c_{1},c_{2}):=\sup_{t\in(0,T_{0})}|c_{1}(t)-c_{2}(t)|+
\sup_{t\in(0,T_{0})}\left\{\frac{t^{2}}{\omega(t)}|c_{1}'(t)-c_{2}'(t)|\right\}.
\label{defn:d-ps1}
\end{equation}

Moreover, a sequence $c_{n}$ converges to some $c_{\infty}$ with respect to this metric if and only if $c_{n}\to c_{\infty}$ uniformly in $[0,T_{0}]$, and for every $\tau\in(0,T_{0})$ it turns out that $c_{n}'\to c_{\infty}'$ uniformly in $[\tau,T_{0}]$.

\item The set $\PSd(T_{0},\mu_{1},\mu_{2},\omega,\psi)$ is a complete metric space with respect to the distance defined by
\begin{equation}
d_{\PSd}(c_{1},c_{2}):=d_{\PSu}(c_{1},c_{2})+
\sup_{t\in(0,T_{0})}\left\{\frac{t^{3}\cdot\exp\left(-\psi(t)\right)}{\omega(t)^{2}}|c_{1}''(t)-c_{2}''(t)|\right\}.
\label{defn:d-ps2}
\end{equation}

Moreover, a sequence $c_{n}$ converges to some $c_{\infty}$ with respect to this metric if and only if $c_{n}\to c_{\infty}$ uniformly in $[0,T_{0}]$, and for every $\tau\in(0,T_{0})$ it turns out that $c_{n}'\to c_{\infty}'$ and $c_{n}''\to c_{\infty}''$ uniformly in $[\tau,T_{0}]$.

\end{itemize}

\end{em}
\end{rmk}


We are now ready to state our results concerning infinite derivative loss. The first one addresses the case with assumptions just on first derivatives.

\begin{thm}[Optimality of Theorem~\ref{thmbibl:c'}]\label{thm:counter-c'}

Let $A$ be an unbounded positive multiplication operator on a Hilbert space $H$.

Let $T_{0}>0$ and $\mu_{2}>\mu_{1}>0$ be real numbers, and let $\omega:(0,T_{0}]\to(0,+\infty)$ be a nonincreasing function such that
\begin{equation}
\lim_{t\to 0^{+}}\omega(t)=+\infty.
\label{hp:omega2infty}
\end{equation}

Let $\PSu(T_{0},\mu_{1},\mu_{2},\omega)$ denote the set of propagation speeds introduced in Definition~\ref{defn:ps}, with the structure of complete metric space induced by the distance (\ref{defn:d-ps1}).

Then the set of propagation speeds $c\in\PSu(T_{0},\mu_{1},\mu_{2},\omega)$ for which problem (\ref{eqn:main})--(\ref{eqn:main-data}) exhibits an infinite derivative loss is residual.

\end{thm}

The second result addresses propagation speeds with assumptions on both first and second order derivatives.

\begin{thm}[Optimality of Theorem~\ref{thmbibl:c''} and Theorem~\ref{thm:main-fdl}]\label{thm:counter-c''}

Let $A$ be an unbounded positive multiplication operator on a Hilbert space $H$.

Let $T_{0}>0$ and $\mu_{2}>\mu_{1}>0$ be real numbers, and let $\omega:(0,T_{0}]\to(0,+\infty)$ and $\psi:(0,T_{0}]\to(0,+\infty)$ be two nonincreasing functions such that
\begin{equation}
\lim_{t\to 0^{+}}\omega(t)=+\infty
\qquad\quad\mbox{and}\quad\qquad
\lim_{t\to 0^{+}}\frac{\omega(t)\psi(t)}{|\log t|}=+\infty.
\label{hp:counter-c''}
\end{equation}

Let $\PSd(T_{0},\mu_{1},\mu_{2},\omega,\psi)$ denote the set of propagation speeds introduced in Definition~\ref{defn:ps}, with the structure of complete metric space induced by the distance (\ref{defn:d-ps2}).

Then the set of propagation speeds $c\in\PSd(T_{0},\mu_{1},\mu_{2},\omega,\psi)$ for which problem (\ref{eqn:main})--(\ref{eqn:main-data}) exhibits an infinite derivative loss is residual.

\end{thm}

\begin{rmk}
\begin{em}

We observe that in Theorem~\ref{thm:counter-c''} we can choose $\omega(t)=|\log t|$ and $\psi(t)$ any nonincreasing function such that $\psi(t)\to +\infty$ as $t\to 0^{+}$. With these choices we obtain an infinite derivative loss for a class of propagations speeds whose first derivatives grow exactly as in (\ref{hp:c'-log}), and whose second derivatives grow just a little bit more than in (\ref{hp:c''}). This improves Theorem~\ref{thmbibl:counter} in showing the optimality of Theorem~\ref{thmbibl:c''}.

\end{em}
\end{rmk}


\setcounter{equation}{0}
\section{Finite derivative loss}\label{sec:well-posed}

In this section we prove Proposition~\ref{prop:main-energy}, from which Theorem~\ref{thm:main-fdl} follows as a standard corollary. Before entering into details, we observe that there is no loss of generality in writing equation (\ref{eqn:u-lambda}) in the form
\begin{equation}
u_{\lambda}''(t)+\lambda^{2}c(t)^{2}u_{\lambda}(t)=0,
\label{eqn:u-lambda-2}
\end{equation}
namely with $c(t)^{2}$ instead of $c(t)$. Indeed, a propagation speed $c(t)$ satisfies (\ref{hp:c-sh}) for some constants $\mu_{1}$ and $\mu_{2}$, and inequalities (\ref{hp:fdl-c'}) and (\ref{hp:fdl-c''}) for some functions $\omega$ and $\psi$ that fulfill (\ref{hp:omega-psi}), if and only if $c(t)^{2}$ satisfies (\ref{hp:c-sh}) with different choices of $\mu_{1}$ and $\mu_{2}$, and inequalities (\ref{hp:fdl-c'}) and (\ref{hp:fdl-c''}) for different choices of the functions $\omega_{1}$ and $\psi_{1}$ that however are equivalent to $\omega$ and $\psi$ up to a multiplicative constant, so that they fulfill again (\ref{hp:omega-psi}) with a possibly different choice of $K_{0}$. 

Therefore, in the sequel we consider the family of ordinary differential equations (\ref{eqn:u-lambda-2}) with $c(t)$ that satisfies (\ref{hp:c-sh}), (\ref{hp:fdl-c'}) and (\ref{hp:fdl-c''}). Writing (\ref{eqn:u-lambda}) in the form (\ref{eqn:u-lambda-2}) allows to express Tarama energy (\ref{defn:et-sqrt}) without square roots, and the absence of square roots simplifies the computation of derivatives. 

\paragraph{\textmd{\textit{Definition and properties of splitting times}}}

For every $\lambda>1$ we set
\begin{equation}
\al:=\frac{\log\lambda}{\lambda},
\qquad\qquad
\bl:=\frac{\log\lambda}{\lambda}\exp\left(\psi\left(1/\lambda\right)\right).
\nonumber
\end{equation}

We can assume, without loss of generality, that there exists $\lambda_{0}>1$ such that
\begin{equation}
\frac{1}{\lambda}<\al<\bl<T_{0}
\qquad
\forall\lambda\geq\lambda_{0}.
\label{hp:al<bl}
\end{equation}

Indeed, the first two inequalities are true as soon as $\lambda>e$, while the third one can be false on a sequence $\lambda_{k}\to +\infty$ only if $\psi(1/\lambda_{k})/\log\lambda_{k}$ is bounded from below by a positive constant. This condition is compatible with (\ref{hp:omega-psi}) only when $\omega(1/\lambda_{k})$ is bounded from above, and hence by monotonicity $\omega(t)$ is bounded from above, namely when $c'(t)$ satisfies the assumptions of Theorem~\ref{thmbibl:c'}. As observed in the introduction, in that case the proof requires only to consider the intervals $[0,\al]$ and $[\al,T_{0}]$, and no assumptions on second derivatives is needed (and indeed we exploit the assumption on $|c''(t)|$ only in the interval $[\bl,T_{0}]$). 

Therefore, in the sequel we assume that (\ref{hp:al<bl}) holds true, and we prove (\ref{th:prop-main}) in the equivalent form 
\begin{equation}
\ek(t)\leq M\ek(0)\exp(\delta\log\lambda)
\qquad
\forall t\in[0,T_{0}],
\label{th=ek}
\end{equation}
where $\ek(t)$ is the Kovaleskyan energy defined in (\ref{defn:ek}). We obtain this inequality by estimating the Kovaleskyan energy in $[0,\al]$, the hyperbolic energy in $[\al,\bl]$, and the Tarama energy in $[\bl,T_{0}]$. 

\paragraph{\textmd{\textit{Energy estimate -- Kovaleskyan region}}}

We show that
\begin{equation}
\ek(t)\leq
\ek(0)\cdot\exp\left((1+\mu_{2}^{2})\log\lambda\right)
\qquad
\forall t\in[0,\al],
\label{th:ek}
\end{equation}
which in particular implies that the inequality in (\ref{th=ek}) holds true for every $t\in[0,\al]$.

Due to (\ref{eqn:u-lambda-2}), the time-derivative of the Kovaleskyan energy is given by
\begin{equation}
\ek'(t)=2\lambda^{2}(1-c(t)^{2})u_{\lambda}(t)u_{\lambda}'(t)
\qquad
\forall t\in(0,T_{0}).
\nonumber
\end{equation}

Since $2\lambda u_{\lambda}(t)u_{\lambda}'(t)\leq\ek(t)$, from (\ref{hp:c-sh}) we deduce that
\begin{equation}
\ek'(t)\leq\lambda(1+\mu_{2}^{2})\ek(t)
\qquad
\forall t\in[0,T_{0}],
\nonumber
\end{equation}
and therefore
\begin{equation}
\ek(t)\leq\ek(0)\exp\left(\lambda(1+\mu_{2}^{2})t\right)
\qquad
\forall t\in(0,T_{0}).
\label{est:ek}
\end{equation}

When $t\leq \al$ it turns out that
\begin{equation}
\lambda(1+\mu_{2}^{2})t\leq
\lambda(1+\mu_{2}^{2})\al=
(1+\mu_{2}^{2})\log\lambda,
\nonumber
\end{equation}
and therefore (\ref{est:ek}) implies (\ref{th:ek}).

\paragraph{\textmd{\textit{Energy estimate -- Hyperbolic region}}}

We show that there exist two positive constants $M_{1}$ and $\delta_{1}$ such that
\begin{equation}
\ek(t)\leq
M_{1}\ek(0)\exp\left(\delta_{1}\log\lambda\right)
\qquad
\forall t\in[\al,\bl],
\label{th:eh}
\end{equation}
which in particular implies that the inequality in (\ref{th=ek}) holds true  for every $t\in[\al,\bl]$.

To this end, we consider the hyperbolic energy $\eh(t)$, which in the case of equation (\ref{eqn:u-lambda-2}) is given by
\begin{equation}
\eh(t):=|u_{\lambda}'(t)|^{2}+\lambda^{2}c(t)^{2}|u_{\lambda}(t)|^{2},
\label{defn:eh-2}
\end{equation}
and we observe that it is equivalent to the Kovaleskyan energy in the sense that
\begin{equation}
\min\{1,\mu_{1}^{2}\}\ek(t)\leq
\eh(t)\leq
\max\{1,\mu_{2}^{2}\}\ek(t)
\qquad
\forall t\in [0,T_{0}].
\label{est:ek-eh}
\end{equation}

Computing the time-derivative of (\ref{defn:eh-2}) we find that
\begin{equation}
\eh'(t)=
2c'(t)c(t)\lambda^{2}u_{\lambda}(t)^{2}\leq
2\frac{|c'(t)|}{c(t)}\eh(t)
\qquad
\forall t\in(0,T_{0}),
\nonumber
\end{equation}
and therefore
\begin{equation}
\eh(t)\leq
\eh(\al)
\exp\left(2\int_{\al}^{\bl}\frac{|c'(s)|}{c(s)}\,ds\right)
\qquad
\forall t\in[\al,\bl].
\label{est:eh-main}
\end{equation}

In order to estimate the integral, we exploit the strict hyperbolicity (\ref{hp:c-sh}), our assumption (\ref{hp:fdl-c'}) and the monotonicity of $\omega$. We obtain that
\begin{equation}
\int_{\al}^{\bl}\frac{|c'(s)|}{c(s)}\,ds\leq
\frac{1}{\mu_{1}}\int_{\al}^{\bl}\frac{\omega(s)}{s}\,ds\leq
\frac{\omega(\al)}{\mu_{1}}\int_{\al}^{\bl}\frac{1}{s}\,ds=
\frac{\omega(\al)}{\mu_{1}}\log\left(\frac{\bl}{\al}\right).
\nonumber
\end{equation}

Since $\al\geq 1/\lambda$, recalling the definition of $\bl$ and assumption (\ref{hp:omega-psi}), we find that 
\begin{equation}
\frac{\omega(\al)}{\mu_{1}}\log\left(\frac{\bl}{\al}\right)\leq
\frac{1}{\mu_{1}}\omega(1/\lambda)\psi(1/\lambda)\leq
\frac{K_{0}}{\mu_{1}}|\log(1/\lambda)|=
\frac{K_{0}}{\mu_{1}}\log\lambda.
\nonumber
\end{equation}

Plugging these estimates into (\ref{est:eh-main}) we deduce that
\begin{equation}
\eh(t)\leq\eh(\al)\exp\left(\frac{2K_{0}}{\mu_{1}}\log\lambda\right)
\qquad
\forall t\in[\al,\bl].
\nonumber
\end{equation}

Finally, thanks to the equivalence (\ref{est:ek-eh}), from (\ref{est:ek}) with $t=\al$ we conclude that
\begin{eqnarray*}
\ek(t) & \leq & \frac{1}{\min\{1,\mu_{1}^{2}\}}\eh(t) \\
& \leq & \frac{1}{\min\{1,\mu_{1}^{2}\}}\eh(\al)\exp\left(\frac{2K_{0}}{\mu_{1}}\log\lambda\right) \\[1ex]
& \leq & \frac{\max\{1,\mu_{2}^{2}\}}{\min\{1,\mu_{1}^{2}\}}\ek(\al)
\exp\left(\frac{2K_{0}}{\mu_{1}}\log\lambda\right)  \\[1ex]
& \leq & \frac{\max\{1,\mu_{2}^{2}\}}{\min\{1,\mu_{1}^{2}\}}\ek(0)
\exp\left(\left\{1+\mu_{2}^{2}+\frac{2K_{0}}{\mu_{1}}\right\}\log\lambda\right)
\end{eqnarray*}
for every $t\in[\al,\bl]$, which proves (\ref{th:eh}).

\paragraph{\textmd{\textit{Energy estimate -- Tarama region}}}

We show that the inequality in (\ref{th=ek}) holds true  for every $t\in[\bl,T_{0}]$. To this end, we consider the Tarama energy, which in the case of equation (\ref{eqn:u-lambda-2}) is given by
\begin{equation}
\et(t):=\frac{|u_{\lambda}'(t)|^{2}}{c(t)}+\lambda^{2}c(t)|u_{\lambda}(t)|^{2}+
\frac{c'(t)^{2}|u_{\lambda}(t)|^{2}}{4c^{3}(t)}+\frac{c'(t)u_{\lambda}(t)u_{\lambda}'(t)}{c^{2}(t)},
\label{defn:et}
\end{equation}

We claim that in the interval $[\bl,T_{0}]$ the Tarama energy is equivalent to the Kovaleskyan one, in the sense that
\begin{equation}
\ep_{0}\ek(t)\leq\et(t)\leq M_{2}\ek(t)
\qquad
\forall t\in[\bl,T_{0}]
\label{est:et-ek}
\end{equation}
for suitable positive real numbers $\ep_{0}$ and $M_{2}$, and that it satisfies the estimate
\begin{equation}
\et(t)\leq\et(\bl)\exp(\delta_{2}\log\lambda)
\qquad
\forall t\in[\bl,T_{0}]
\label{est:et}
\end{equation}
for another positive real number $\delta_{2}$. If we prove these claims, then from (\ref{est:et-ek}), (\ref{est:et}) and (\ref{th:eh}) with $t=\bl$ we conclude that
\begin{eqnarray*}
\ek(t) & \leq &  \frac{1}{\ep_{0}}\et(t) \\[0.5ex]
& \leq & \frac{1}{\ep_{0}}\et(\bl)\exp(\delta_{2}\log\lambda) \\[0.5ex]
& \leq & \frac{M_{2}}{\ep_{0}}\ek(\bl)\exp(\delta_{2}\log\lambda) \\[0.5ex]
& \leq & 
\frac{M_{1}M_{2}}{\ep_{0}}\ek(0)\exp\left\{\left(\delta_{1}+\delta_{2}\right)\log\lambda\right\}
\end{eqnarray*}
for every $t\in[\bl,T_{0}]$, which completes the proof. Therefore, in the sequel we limit ourselves to proving (\ref{est:et-ek}) and (\ref{est:et}).

\subparagraph{\textmd{\textit{Equivalence of energies}}}

To begin with, we show the key estimate
\begin{equation}
|c'(t)|\leq K_{0}\lambda
\qquad
\forall t\in[\bl,T_{0}].
\label{est:c'/c2}
\end{equation}

To this end, from (\ref{hp:fdl-c'}) and the monotonicity of $\omega$ we obtain that
\begin{equation}
|c'(t)|\leq
\frac{\omega(t)}{t}\leq
\frac{\omega(\bl)}{\bl}
\qquad
\forall t\in[\bl,T_{0}].
\nonumber
\end{equation}

Since $\bl\geq 1/\lambda$, exploiting again the monotonicity of $\omega$ and assumption (\ref{hp:omega-psi}), we deduce that
\begin{equation}
\omega(\bl)\leq\omega(1/\lambda)\leq K_{0}|\log(1/\lambda)|= K_{0}\log\lambda,
\label{est:omega-bl}
\end{equation}
and therefore from the definition of $\bl$ we conclude that
\begin{equation}
|c'(t)|\leq
\frac{\omega(\bl)}{\bl}\leq
K_{0}\log\lambda\cdot\frac{\lambda}{\log\lambda}\exp(-\psi(1/\lambda))\leq
K_{0}\lambda,
\nonumber
\end{equation}
which proves (\ref{est:c'/c2}).

In order to prove the estimate from below in (\ref{est:et-ek}), we actually show that there exists $\ep_{1}>0$ such that
\begin{equation}
\et(t)\geq\frac{\ep_{1}}{c(t)}\eh(t)
\qquad
\forall t\in[\bl,T_{0}],
\label{est:et-ceh}
\end{equation}
which in turn implies the required estimate because
\begin{equation}
\frac{\ep_{1}}{c(t)}\eh(t)\geq
\frac{\ep_{1}}{\mu_{2}}\eh(t)\geq
\frac{\ep_{1}}{\mu_{2}}\min\{1,\mu_{1}^{2}\}\ek(t).
\nonumber
\end{equation}

In order to prove (\ref{est:et-ceh}), we choose a real number $\ep_{1}\in(0,1)$ such that
\begin{equation}
\ep_{1}K_{0}^{2}\leq 4(1-\ep_{1})^{2}\mu_{1}^{4}
\nonumber
\end{equation}
(this inequality is true when $\ep_{1}$ is small enough). At this point (\ref{est:et-ceh}) is equivalent to
\begin{equation}
(1-\ep_{1})\frac{|u_{\lambda}'(t)|^{2}}{c(t)}+
\left(\frac{c'(t)^{2}}{4c(t)^{3}}+(1-\ep_{1})c(t)\lambda^{2}\right)|u_{\lambda}(t)|^{2}+
\frac{c'(t)}{c(t)^{2}}u_{\lambda}(t)u_{\lambda}'(t)\geq 0.
\nonumber
\end{equation}

The left-hand side is a quadratic form in the variables $u_{\lambda}(t)$ and $u_{\lambda}'(t)$. The coefficients of the quadratic terms are positive because $\ep_{1}<1$. Therefore, the form is nonnegative if
\begin{equation}
\frac{(1-\ep_{1})}{c(t)}\cdot\left(\frac{c'(t)^{2}}{4c(t)^{3}}+(1-\ep_{1})c(t)\lambda^{2}\right)\geq
\frac{1}{4}\frac{c'(t)^{2}}{c(t)^{4}},
\nonumber
\end{equation}
which in turn is equivalent to
\begin{equation}
\ep_{1}c'(t)^{2}\leq 4(1-\ep_{1})^{2}c(t)^{4}\lambda^{2},
\nonumber
\end{equation}
and this inequality follows from (\ref{est:c'/c2}) because
\begin{equation}
\ep_{1}c'(t)^{2}\leq 
\ep_{1}K_{0}^{2}\lambda^{2}\leq
4(1-\ep_{1})^{2}\mu_{1}^{4}\lambda^{2}\leq
4(1-\ep_{1})^{2}c(t)^{4}\lambda^{2}.
\nonumber
\end{equation}

It remains to prove the estimate from above in (\ref{est:et-ek}). Due to the strict hyperbolicity condition (\ref{hp:c-sh}), the first two terms in (\ref{defn:et}) can be estimated as
\begin{equation}
\frac{|u_{\lambda}'(t)|^{2}}{c(t)}\leq\frac{1}{\mu_{1}}\ek(t)
\qquad\quad\mbox{and}\qquad\quad
\lambda^{2}c(t)|u_{\lambda}(t)|^{2}\leq\mu_{2}\ek(t).
\nonumber
\end{equation}

In order to estimate the third and fourth term, we exploit again (\ref{est:c'/c2}) and we deduce that
\begin{equation}
\frac{c'(t)^{2}}{4c(t)^{3}}|u_{\lambda}(t)|^{2}\leq
\frac{K_{0}^{2}}{4\mu_{1}^{3}}\,\lambda^{2}|u_{\lambda}(t)|^{2}\leq
\frac{K_{0}^{2}}{4\mu_{1}^{3}}\,\ek(t),
\nonumber
\end{equation}
and
\begin{equation}
\frac{c'(t)}{c(t)^{2}}u_{\lambda}(t)u_{\lambda}'(t)\leq
\frac{K_{0}}{\mu_{1}^{2}}\,\lambda |u_{\lambda}(t)|\cdot|u_{\lambda}'(t)|\leq
\frac{K_{0}}{2\mu_{1}^{2}}\,\ek(t).
\nonumber
\end{equation}

\subparagraph{\textmd{\textit{Differential inequality for the Tarama energy}}}

In order to prove estimate (\ref{est:et}), we compute the time-derivative of (\ref{defn:et}), and we obtain that
\begin{equation}
\et'(t)=\left(\frac{|u_{\lambda}(t)|^{2}}{2}\frac{c'(t)}{c(t)^{3}}+\frac{u_{\lambda}(t)u_{\lambda}'(t)}{c(t)^{2}}\right)\left(c''(t)-\frac{3}{2}\frac{c'(t)^{2}}{c(t)}\right)
\qquad
\forall t\in(0,T_{0}).
\label{eqn:et'}
\end{equation}

When $t\geq\bl$, we can estimate the first two terms by exploiting once again (\ref{est:c'/c2}), and we obtain that
\begin{equation}
\frac{|u_{\lambda}(t)|^{2}}{2}\frac{|c'(t)|}{c(t)^{3}}\leq
\frac{K_{0}}{2\mu_{1}^{3}}\,\lambda |u_{\lambda}(t)|^{2}\leq
\frac{K_{0}}{2\mu_{1}^{3}}\cdot\frac{1}{\lambda}\ek(t)\leq
\frac{K_{0}}{2\mu_{1}^{3}\ep_{0}}\cdot\frac{1}{\lambda}\et(t),
\nonumber
\end{equation}
and
\begin{equation}
\frac{|u_{\lambda}(t)u_{\lambda}'(t)|}{c(t)^{2}}=
\frac{1}{\lambda c(t)^{2}}\cdot\lambda|u_{\lambda}(t)u_{\lambda}'(t)|\leq
\frac{1}{2\mu_{1}^{2}}\cdot\frac{1}{\lambda}\ek(t)\leq
\frac{1}{2\mu_{1}^{2}\ep_{0}}\cdot\frac{1}{\lambda}\et(t).
\nonumber
\end{equation}

Plugging these estimates into (\ref{eqn:et'}), we deduce that there exists a positive constant $M_{3}$ such that
\begin{equation}
\et'(t)\leq\frac{M_{3}}{\lambda}\et(t)\left(|c''(t)|+|c'(t)|^{2}\right)
\qquad
\forall t\in(\bl,T_{0}),
\nonumber
\end{equation}
and therefore
\begin{equation}
\et(t)\leq
\et(\bl)\exp\left(\frac{M_{3}}{\lambda}\int_{\bl}^{T_{0}}\left\{|c''(s)|+|c'(s)|^{2}\right\}\,ds\right)
\qquad
\forall t\in[\bl,T_{0}].
\label{est:et-int}
\end{equation}

It remains to estimate the last integral, and this is the point where the assumption on the second derivative of the propagation speed comes into play. To begin with, comparing (\ref{hp:fdl-c'}) and (\ref{hp:fdl-c''}) we observe that it is enough to bound the integral of $|c''(s)|$. Now we exploit assumption (\ref{hp:fdl-c''}), the monotonicity of $\omega$ and $\psi$, and inequality (\ref{est:omega-bl}), and we deduce that
\begin{eqnarray*}
\int_{\bl}^{T_{0}}|c''(s)|\,ds 
& \leq & \int_{\bl}^{T_{0}}\frac{\omega(s)^{2}}{s^{2}}\exp(\psi(s))\,ds  \\[0.5ex]
& \leq & \omega(\bl)^{2}\exp(\psi(\bl))\int_{\bl}^{T_{0}}\frac{1}{s^{2}}\,ds  \\[1ex]
& \leq & K_{0}^{2}(\log\lambda)^{2}\exp(\psi(\bl))\frac{1}{\bl}  \\[1ex]
& = & K_{0}^{2}(\log\lambda)^{2}\exp(\psi(\bl))\cdot\frac{\lambda}{\log\lambda}\exp\left(-\psi(1/\lambda)\strut\right) \\[0.5ex]
& \leq & K_{0}^{2}\,\lambda\log\lambda,
\end{eqnarray*}
where in the last inequality we exploited that $\bl\geq 1/\lambda$, and hence $\psi(\bl)\leq\psi(1/\lambda)$. 

Plugging this estimate into (\ref{est:et-int}) we obtain (\ref{est:et}) with $\delta_{2}:=M_{3}K_{0}^{2}$.


\setcounter{equation}{0}
\section{Counterexamples}\label{sec:counterexamples}

This section is devoted to the construction of the counterexamples that are needed for a proof of Theorem~\ref{thm:counter-c'} and Theorem~\ref{thm:counter-c''}. The strategy is the following.
\begin{itemize}

\item In subsection~\ref{sec:activators} we introduce special classes of propagation speeds, which we call universal and asymptotic activators. We show that the existence of families of asymptotic activators converging to elements of a dense set implies the existence of a residual set of universal activators, and that problem (\ref{eqn:main})--(\ref{eqn:main-data}) exhibits an infinite derivative loss whenever the propagation speed $c(t)$ is a universal activator. This reduces the construction of counterexamples to the existence of such families of asymptotic activators.

\item  In subsection~\ref{sec:icf} we identify a suitable dense set of propagation speeds, consisting of what we call initially constant functions.

\item  In subsection~\ref{sec:as-act-icf} we present a general parametric construction that, under suitable assumptions on the parameters, produces the required family of asymptotic activators.

\item  In subsection~\ref{sec:parameters} we show that, under the assumptions of Theorem~\ref{thm:counter-c'} and Theorem~\ref{thm:counter-c''}, we can fulfill the required assumptions on the parameters.

\item  Finally, in subsection~\ref{sec:iteration} we show that our ingredients can be combined in a different way in order to produce an ``explicit'' counterexample through an iteration procedure in the spirit of~\cite{dgcs}.

\end{itemize}

In the sequel we consider solutions to the family of ordinary differential equations 
\begin{equation}
u_{\lambda}''(t)+\lambda^{2}\cl(t)u_{\lambda}(t)=0,
\label{eqn:uc-lambda}
\end{equation}
with initial data
\begin{equation}
u_{\lambda}(0)=0,
\qquad
u_{\lambda}'(0)=1.
\label{data:ode-lambda}
\end{equation}

We observe that (\ref{eqn:uc-lambda}) is of the form (\ref{eqn:u-lambda}), but here the propagation speed depends on the parameter $\lambda$. When the propagation speed is fixed, but we deal with a sequence $\{\lambda_{n}\}$ of parameters, we avoid double indices by writing equation (\ref{eqn:u-lambda}) in the form
\begin{equation}
u_{n}''(t)+\lambda_{n}^{2}c(t)u_{n}(t)=0,
\label{eqn:ode-lambda-n}
\end{equation}
and again we consider initial data of the form
\begin{equation}
u_{n}(0)=0,
\qquad
u_{n}'(0)=1.
\label{data:ode-lambda-n}
\end{equation}

We also fix, once for all, a cutoff function $\theta:\re\to\re$ of class $C^{\infty}$ satisfying
\begin{equation}
\fbox{$
\begin{array}{@{\,}c@{\qquad}c}
\theta(\sigma)=0 & \forall\sigma\leq 0 \\[0.5ex]
\theta(\sigma)=1 & \forall\sigma\geq 1 \\[0.5ex]
0\leq\theta(\sigma)\leq 1 & \forall\sigma\in[0,1]
\end{array}
$}
\label{defn:cutoff}
\end{equation}


\subsection{Asymptotic and universal activators}\label{sec:activators}

Let us introduce our notion of activators (compare with~\cite{gg:dgcs-critical}).

\begin{defn}[Universal activators of a sequence]\label{defn:un-act}
\begin{em}

Let $T_{0}$ be a positive real number, let $\phi:(0,+\infty)\to(0,+\infty)$ be a function, and let $\{\lambda_{n}\}$ be a sequence of positive real numbers such that $\lambda_{n}\to +\infty$.

A \emph{universal activator} of the sequence $\{\lambda_{n}\}$ with rate $\phi$ is a propagation speed $c\in L^{1}((0,T_{0}))$ such that the corresponding sequence $\{u_{n}(t)\}$ of solutions to (\ref{eqn:ode-lambda-n})--(\ref{data:ode-lambda-n}) satisfies
\begin{equation}
\limsup_{n\to +\infty}\left(|u_{n}'(t)|^{2}+\lambda_{n}^{2}|u_{n}(t)|^{2}\right)
\exp\left(-\phi(\lambda_{n})\strut\right)\geq 1
\qquad
\forall t\in(0,T_{0}].
\label{th:defn-un-act}
\end{equation}

\end{em}
\end{defn}

\begin{defn}[Asymptotic activators]\label{defn:as-act}
\begin{em}

Let $T_{0}$ be a positive real number, and let $\phi:(0,+\infty)\to(0,+\infty)$ be a function. 

A \emph{family of asymptotic activators} with rate $\phi$ is a family of propagation speeds $\{c_{\lambda}(t)\}\subseteq L^{1}((0,T_{0}))$ with the property that, for every $\delta\in(0,T_{0})$, there exist two positive constants $M_{\delta}$ and $\lambda_{\delta}$ such that the corresponding family $\{u_{\lambda}(t)\}$ of solutions to (\ref{eqn:uc-lambda})--(\ref{data:ode-lambda}) satisfies
\begin{equation}
|u_{\lambda}'(t)|^{2}+\lambda^{2}|u_{\lambda}(t)|^{2}\geq
M_{\delta}\exp\left(2\phi(\lambda)\right)
\qquad
\forall t\in[\delta,T_{0}],\quad\forall\lambda\geq\lambda_{\delta}.
\label{eqn:asympt-activ}
\end{equation}

\end{em}
\end{defn}

We point out that (\ref{th:defn-un-act}) is a qualitative statement, apparently weaker than the quantitative version (\ref{eqn:asympt-activ}). On the contrary, (\ref{th:defn-un-act}) is stronger because it concerns the family of equations (\ref{eqn:ode-lambda-n}), where the propagation speed is the same for every $n$, while (\ref{eqn:asympt-activ}) concerns the family of equations (\ref{eqn:uc-lambda}), where the propagation speed depends on $\lambda$.


The following statement clarifies the crucial connection between universal activators and infinite derivative loss.

\begin{prop}[Universal activators lead to infinite derivative loss]\label{prop:ua2idl}

Let $H$ be a Hilbert space, and let $A$, $\{e_{n}\}$ and $\{\lambda_{n}\}$ be as in Definition~\ref{defn:upmo}. Let us assume, without loss of generality (since we can always consider a subsequence) that
\begin{equation}
\sum_{n=1}^{\infty}\frac{1}{\lambda_{n}}<+\infty.
\label{hp:series-ln}
\end{equation}

Let $T_{0}$ be a positive real number, let $\phi:(0,+\infty)\to(0,+\infty)$ be a function such that
\begin{equation}
\lim_{\lambda\to +\infty}\frac{\phi(\lambda)}{\log\lambda}= +\infty,
\label{hp:phi-log}
\end{equation}
and let $c\in L^{1}((0,T_{0}))$ be a universal activator of the sequence $\{\lambda_{n}\}$ with rate $\phi$. 

Then solutions to problem (\ref{eqn:main})--(\ref{eqn:main-data}) exhibits an infinite derivative loss according to Definition~\ref{defn:regularity}.

\end{prop}

\begin{proof}

For every positive integer $n$, let us set
\begin{equation}
a_{n}:=\exp\left(-\frac{\phi(\lambda_{n})}{4}\right),
\nonumber
\end{equation}
and let us consider problem (\ref{eqn:main})--(\ref{eqn:main-data}) with initial data
\begin{equation}
u_{0}:=0,
\qquad\qquad
u_{1}:=\sum_{n=1}^{\infty}a_{n}e_{n}.
\nonumber
\end{equation}

Due to (\ref{hp:phi-log}), for every $\alpha\in\re$ it turns out that
\begin{equation}
a_{n}^{2}\lambda_{n}^{2\alpha}=
\exp\left(-\frac{\phi(\lambda_{n})}{2}+2\alpha\log\lambda_{n}\right)\leq
\frac{1}{\lambda_{n}}
\nonumber
\end{equation}
when $n$ is large enough, and therefore from (\ref{hp:series-ln}) it follows that
\begin{equation}
\sum_{n=1}^{\infty}a_{n}^{2}\lambda_{n}^{2\alpha}<+\infty
\qquad
\forall\alpha\in\re,
\nonumber
\end{equation}
which proves that $u_{1}\in D(A^{\infty})$. 

On the other hand, the unique solution to (\ref{eqn:main})--(\ref{eqn:main-data}) is given by (a priori this series converges just in the sense of ultradistributions) 
\begin{equation}
u(t):=\sum_{n=1}^{\infty}a_{n}u_{n}(t)e_{n}
\qquad
\forall t\in[0,T_{0}],
\nonumber
\end{equation}
where $u_{n}(t)$ is the sequence of solutions to (\ref{eqn:ode-lambda-n})--(\ref{data:ode-lambda-n}). Due to (\ref{hp:phi-log}), for every $\beta\in\re$ it turns out that
\begin{eqnarray*}
a_{n}^{2}\left(|u_{n}'(t)|^{2}+\lambda_{n}^{2}|u_{n}(t)|^{2}\right)\lambda_{n}^{-2\beta} & = &
\left(|u_{n}'(t)|^{2}+\lambda_{n}^{2}|u_{n}(t)|^{2}\right)\exp(-\phi(\lambda_{n}))\cdot \\[0.5ex]
& & \mbox{}\cdot\exp\left(\phi(\lambda_{n})-\frac{\phi(\lambda_{n})}{2}-2\beta\log\lambda_{n}\right) \\[0.5ex]
& \geq & \left(|u_{n}'(t)|^{2}+\lambda_{n}^{2}|u_{n}(t)|^{2}\right)\exp(-\phi(\lambda_{n}))
\end{eqnarray*}
when $n$ is large enough. Since $c(t)$ is a universal activator of the sequence $\{\lambda_{n}\}$ with rate $\phi$, from the last inequality and (\ref{th:defn-un-act}) it follows that
\begin{equation}
\limsup_{n\to +\infty}a_{n}^{2}\left(|u_{n}'(t)|^{2}+\lambda_{n}^{2}|u_{n}(t)|^{2}\right)
\lambda_{n}^{-2\beta}\geq 1\qquad
\forall\beta\in\re,
\nonumber
\end{equation}
and therefore 
\begin{equation}
\sum_{n=1}^{\infty}
a_{n}^{2}\left(|u_{n}'(t)|^{2}+\lambda_{n}^{2}|u_{n}(t)|^{2}\right)\lambda_{n}^{-2\beta}
=+\infty
\qquad
\forall\beta\in\re.
\nonumber
\end{equation}

This proves that $(u(t),u'(t))\not\in D(A^{-\beta+1/2})\times D(A^{-\beta})$ for every $\beta\in\re$.
\end{proof}


Families of asymptotic activators are the basic tool in the construction of universal activators. Before entering into details, we recall a well-know result concerning the dependence of solutions to (\ref{eqn:u-lambda}) on the propagation speed $c(t)$, when $\lambda$ is considered a fixed parameter.

\begin{lemma}[Continuous dependence on the propagation speed]\label{lemma:cont-dep}

Let $T_{0}$ and $\lambda$ be fixed positive real numbers. Let us consider problem (\ref{eqn:u-lambda})--(\ref{data:ode-lambda}) with a sequence of propagation speeds $\{c_{n}(t)\}$ converging to some $c_{\infty}(t)$ in $L^{1}((0,T_{0}))$, and let $\{u_{n}(t)\}$ and $u_{\infty}(t)$ denote the corresponding solutions.

Then $u_{n}(t)\to u_{\infty}(t)$ and $u_{n}'(t)\to u_{\infty}'(t)$ uniformly in $[0,T_{0}]$.

\end{lemma}

In the following result we pass from asymptotic to universal activators. This is the point where Baire category theorem discloses its power.

\begin{prop}[From asymptotic to universal activators]\label{prop:as2un}

Let $\phi:(0,+\infty)\to(0,+\infty)$ be a function such that $\phi(\lambda)\to +\infty$ as $\lambda\to +\infty$. 

Let $T_{0}$ be a positive real number, and let $\PS\subseteq L^{1}((0,T_{0}))$ be a subset that is a complete metric space with respect to some distance $d_{\PS}$ with the property that convergence with respect to $d_{\PS}$ implies convergence in $L^{1}((0,T_{0}))$.

Let us assume that there exists a subset $\mathcal{D}\subseteq\PS$, dense with respect to the distance $d_{\PS}$, such that for every $c\in\mathcal{D}$ there exists a family of asymptotic activators $\{c_{\lambda}\}\subseteq\PS$ with rate $\phi$ such that $c_{\lambda}\to c$, always with respect to the distance $d_{\PS}$.

Then, for every unbounded sequence $\{\lambda_{n}\}$ of positive real numbers, the set of elements in $\PS$ that are universal activators of the sequence $\{\lambda_{n}\}$ with rate $\phi$ is residual in $\PS$ (and in particular nonempty).

\end{prop}

\begin{proof}

The argument is exactly the same that appears in the proof of~\cite[Theorem~3.2]{gg:residual} and~\cite[Theorem~4.5]{gg:dgcs-critical}, and therefore here we limit ourselves to sketching the main steps for the convenience of the reader.

Let $\mathcal{C}$ denote the set of elements in $\PS$ that are not universal activators of the sequence $\{\lambda_{n}\}$ with rate $\phi$. We need to show that $\mathcal{C}$ is the countable union of closed subsets of $\PS$ with empty interior. 

To this end, we begin by observing that every element $c\in\mathcal{C}$ satisfies 
\begin{equation}
\exists t_{0}\in(0,T_{0}]
\qquad
\limsup_{n\to+\infty}
\left(|u_{n}'(t_{0})|^{2}+\lambda_{n}^{2}|u_{n}(t_{0})|^{2}\right)
\exp\left(-\phi(\lambda_{n})\right)<1,
\nonumber
\end{equation}
or equivalently
\begin{multline}
\exists t_{0}\in(0,T_{0}]\quad
\exists\ep_{0}>0\quad
\exists n_{0}\in\n\quad 
\forall n\geq n_{0}  \\[0.5ex]
\left(|u_{n}'(t_{0})|^{2}+\lambda_{n}^{2}|u_{n}(t_{0})|^{2}\right)
\exp\left(-\phi(\lambda_{n})\right)\leq 1-\ep_{0},
\nonumber
\end{multline}
where of course $t_{0}$, $n_{0}$ and $\ep_{0}$ might depend on $c$. In order to make this property more quantitative, for every positive integer $k$ we introduce the class $\mathcal{C}_{k}$ of all elements $c\in\PS$ such that
\begin{equation}
\exists t_{0}\in\left[\frac{1}{k},T_{0}\right]\quad 
\forall n\geq k
\qquad
\left(|u_{n}'(t_{0})|^{2}+\lambda_{n}^{2}|u_{n}(t_{0})|^{2}\right)
\exp\left(-\phi(\lambda_{n})\right)\leq 
1-\frac{1}{k}.
\nonumber
\end{equation}

Due to Lemma~\ref{lemma:cont-dep}, the set $\mathcal{C}_{k}$ is closed with respect to convergence in $L^{1}((0,T_{0}))$, and hence also with respect to convergence with respect to the distance $d_{\PS}$. Moreover, $\mathcal{C}$ is the union of all $\mathcal{C}_{k}$'s.

It remains to show that $\mathcal{C}_{k}$ has empty interior for every positive integer $k$. Let us assume by contradiction that this is not the case, and therefore there exist $k_{0}\in\n$, $c_{0}\in\mathcal{C}_{k_{0}}$ and $r_{0}>0$ such that $B(c_{0},r_{0})\subseteq\mathcal{C}_{k_{0}}$, where $B(c_{0},r_{0})$ denotes the ball in $\PS$ with center in $c_{0}$ and radius $r_{0}$ with respect to the distance $d_{\PS}$.

Since $\mathcal{D}$ is dense, up to reducing the radius we can always assume that $c_{0}\in\mathcal{D}$. Now we exploit the existence of a family of asymptotic activators $\{c_{\lambda}\}$ with rate $\phi$ that converges to $c$ in $\PS$. Due to this convergence, we know that $c_{\lambda}\in B(c_{0},r_{0})\subseteq\mathcal{C}_{k_{0}}$ when $\lambda$ is large enough, and therefore when $\lambda_{n}$ is large enough we know that
\begin{equation}
\exists t_{0n}\in[1/k_{0},T_{0}]
\qquad
\left(|u_{n}'(t_{0n})|^{2}+\lambda_{n}^{2}|u_{n}(t_{0n})|^{2}\right)
\exp\left(-\phi(\lambda_{n})\right)\leq 
1-\frac{1}{k_{0}}.
\label{absurd}
\end{equation}

On the other hand, from property (\ref{eqn:asympt-activ}) in the definition of asymptotic activators, applied with $\delta:=1/k_{0}$, we know that
\begin{equation}
\forall t\in[1/k_{0},T_{0}]
\qquad
|u_{n}'(t)|^{2}+\lambda_{n}^{2}|u_{n}(t)|^{2}\geq
M_{1/k_{0}}\exp\left(2\phi(\lambda_{n})\right)
\nonumber
\end{equation}
when $\lambda_{n}$ is large enough, and this contradicts (\ref{absurd}) when $\lambda_{n}\to+\infty$.
\end{proof}


\subsection{Initially constant functions}\label{sec:icf}

In the previous subsection we have reduced the search of propagation speeds generating infinite derivative loss to the search of families of asymptotic activators converging to elements of a suitable dense set. Now we identify this dense subset. 

\begin{defn}[Initially constant functions]\label{defn:icf}
\begin{em}

Let $T_{0}>0$ and $\mu_{2}>\mu_{1}>0$ be real numbers.

We call $\mathcal{D}(T_{0},\mu_{1},\mu_{2})$ the set of functions $c:[0,T_{0}]\to[\mu_{1},\mu_{2}]$ for which there exists two real numbers $T_{1}\in(0,T_{0})$ and $\mu_{3}\in(\mu_{1},\mu_{2})$ such that $c(t)=\mu_{3}$ for every $t\in[0,T_{1}]$.

\end{em}
\end{defn}

We stress that the ``initial constant'' $\mu_{3}$ is strictly in between $\mu_{1}$ and $\mu_{2}$. 

The result is that initially constant functions are dense in the classes of propagation speeds that we are interested in.

\begin{prop}[Density of initially constant functions]\label{prop:icf}

Let $T_{0}>0$ and $\mu_{2}>\mu_{1}>0$ be real numbers, and let $\omega:(0,T_{0}]\to(0,+\infty)$ and $\psi:(0,T_{0}]\to(0,+\infty)$ be two nonincreasing functions, with $\omega(t)\to +\infty$ as $t\to 0^{+}$.

Let us consider the classes of propagation speeds of Definition~\ref{defn:ps}, with the structure of complete metric space introduced in Remark~\ref{rmk:cms}. Let $\mathcal{D}(T_{0},\mu_{1},\mu_{2})$ denote the class of initially constant functions of Definition~\ref{defn:icf}.

Then it turns out that
\begin{enumerate}
\renewcommand{\labelenumi}{(\arabic{enumi})}

\item  $\mathcal{D}(T_{0},\mu_{1},\mu_{2})\cap\PSu(T_{0},\mu_{1},\mu_{2},\omega)$ is dense in $\PSu(T_{0},\mu_{1},\mu_{2},\omega)$,

\item  $\mathcal{D}(T_{0},\mu_{1},\mu_{2})\cap\PSd(T_{0},\mu_{1},\mu_{2},\omega,\psi)$ is dense in $\PSd(T_{0},\mu_{1},\mu_{2},\omega,\psi)$.

\end{enumerate}

\end{prop}

\begin{proof}

For the sake of shortness, in the sequel we omit the dependence on the parameters, and we denote the function spaces by $\mathcal{D}$, $\PSu$ and $\PSd$.

\paragraph{\textmd{\textit{Statement (1)}}}

We need to show that, for every $c_{0}\in\PSu$ and every $r>0$, there exists $c\in\mathcal{D}\cap\PSu$ such that $d_{\PSu}(c,c_{0})\leq r$. Our construction involves two steps.

In the first step we show that there exists $\chat\in\PSu$ such that 
\begin{equation}
d_{\PSu}(\chat,c_{0})\leq r/2,
\label{est:r/2-1}
\end{equation}
and such that $\chat$ does not saturate the inequalities in the definition of $\PSu$, namely there exists $\ep_{0}>0$ such that
\begin{equation}
(1+\ep_{0})\mu_{1}\leq\chat(t)\leq (1-\ep_{0})\mu_{2}
\qquad
\forall t\in[0,T_{0}],
\label{non-sat-c}
\end{equation}
and
\begin{equation}
|\chat'(t)|\leq(1-\ep_{0})\frac{\omega(t)}{t}
\qquad
\forall t\in(0,T_{0}].
\label{non-sat-c'}
\end{equation}

To this end, for every $\ep>0$ we consider the affine transformation $A_{\ep}:\re\to\re$ such that $A_{\ep}(\mu_{1})=(1+\ep)\mu_{1}$ and  $A_{\ep}(\mu_{2})=(1-\ep)\mu_{2}$, and we observe that, when $\ep$ is small enough, the composition $A_{\ep}(c(t))$ satisfies (\ref{non-sat-c}), (\ref{non-sat-c'}) and (\ref{est:r/2-1}). 

In the second step we show that there exists $c\in\mathcal{D}\cap\PSu$ such that
\begin{equation}
d_{\PSu}(c,\chat)\leq r/2.
\label{est:r/2-2}
\end{equation}

To this end, we consider a cutoff function $\theta:\re\to\re$ of class $C^{\infty}$ satisfying (\ref{defn:cutoff}), and for every $\delta\in(0,T_{0}/2)$ we set
\begin{equation}
\cd(t):=\chat(\delta)+\theta\left(\frac{t-\delta}{\delta}\right)\left(\chat(t)-\chat(\delta)\right)
\qquad
\forall t\in[0,T_{0}].
\label{defn:cd}
\end{equation}

We claim that $\cd\in\mathcal{D}\cap\PSu$ if $\delta$ is small enough, and that $d_{\PSu}(\cd,\chat)\to 0$ as $\delta\to 0^{+}$. If we prove these claims, then inequality (\ref{est:r/2-2}) is satisfied for $\delta$ small enough, and taking (\ref{est:r/2-1}) into account this means that any such $\cd$ is the required function.

Due to the properties of $\theta$, the function $\cd$ coincides with the constant $\chat(\delta)$ in the interval $[0,\delta]$. Since $\chat(\delta)\in(\mu_{1},\mu_{2})$, this proves that $\cd\in\mathcal{D}$. 

When $\delta$ is small enough we show that $\cd\in\PSu$, namely that
\begin{equation}
\mu_{1}\leq\cd(t)\leq\mu_{2}
\qquad
\forall t\in[0,T_{0}],
\label{est:cd}
\end{equation}
and
\begin{equation}
|\cd'(t)|\leq\frac{\omega(t)}{t}
\qquad
\forall t\in(0,T_{0}].
\label{est:cd'}
\end{equation}

To begin with, we set
\begin{equation}
\Gamma_{\delta}:=\max\left\{|\chat(t)-\chat(\delta)|:t\in[0,2\delta]\right\},
\label{defn:gamma-delta}
\end{equation}
and we observe that $\Gamma_{\delta}\to 0$ as $\delta\to 0^{+}$ because $c_{*}$ is continuous up to $t=0$. Since $\cd(t)=\chat(t)$ for every $t\in[2\delta,T_{0}]$, and 
\begin{equation}
|\cd(t)-\chat(t)|=
\left|1-\theta\left(\frac{t-\delta}{\delta}\right)\right|\cdot|\chat(t)-\chat(\delta)|\leq
\Gamma_{\delta}
\qquad
\forall t\in[0,2\delta],
\nonumber
\end{equation}
this proves that
\begin{equation}
\cd(t)\to\chat(t)
\qquad
\mbox{uniformly in }[0,T_{0}].
\nonumber
\end{equation}

Since $\chat$ satisfies the stronger inequality (\ref{non-sat-c}), it follows that $\cd$ satisfies (\ref{est:cd}) when $\delta$ is small enough.

As for the derivative, we observe that (\ref{est:cd'}) holds true almost for free in the interval $[0,\delta]$, where $\cd'(t)=0$, and in the interval $[2\delta,T_{0}]$, where $\cd'(t)$ coincides with $\chat'(t)$, which in turns satisfies the stronger estimate (\ref{non-sat-c'}). In the interval $[\delta,2\delta]$ the derivative of $\cd$ is given by
\begin{equation}
\cd'(t)=\theta'\left(\frac{t-\delta}{\delta}\right)\frac{1}{\delta}(\chat(t)-\chat(\delta))+
\theta\left(\frac{t-\delta}{\delta}\right)\chat'(t).
\nonumber
\end{equation}

From (\ref{defn:gamma-delta}) and the monotonicity of $\omega(t)$ it follows that
\begin{eqnarray*}
\left|\theta'\left(\frac{t-\delta}{\delta}\right)\frac{1}{\delta}(\chat(t)-\chat(\delta))\right|
& \leq & \|\theta'\|_{\infty}\frac{\Gamma_{\delta}}{\delta}  \\
& = & \frac{\|\theta'\|_{\infty}\Gamma_{\delta}}{\delta}\cdot
\frac{t}{\omega(t)}\cdot\frac{\omega(t)}{t} \\
& \leq & \frac{2\|\theta'\|_{\infty}\Gamma_{\delta}}{\omega(2\delta)}\cdot
\frac{\omega(t)}{t},
\end{eqnarray*}
and therefore, taking (\ref{non-sat-c'}) into account, we deduce that
\begin{equation}
|\cd'(t)|\leq\left\{
\frac{2\|\theta'\|_{\infty}\Gamma_{\delta}}{\omega(2\delta)}+1-\ep_{0}
\right\}\frac{\omega(t)}{t}
\qquad
\forall t\in[\delta,2\delta].
\nonumber
\end{equation}

Since $\Gamma_{\delta}\to 0$ and $\omega(2\delta)\to +\infty$ as $\delta\to 0^{+}$, inequality (\ref{est:cd'}) holds true also in $[\delta,2\delta]$ provided that $\delta$ is small enough. 

It remains to show that $\cd\to\chat$ in $\PSu$. Since we already now that $\cd\to\chat$ uniformly in $[0,T_{0}]$, we need only to check that
\begin{equation}
\cd(t)\to\chat(t)
\qquad
\mbox{uniformly in }[\tau,T_{0}]
\nonumber
\end{equation}
for every fixed $\tau\in(0,T_{0})$. This is true almost for free because, once that $\tau$ has been fixed, we know that $\cd'(t)=\chat(t)$ in $[\tau,T_{0}]$ when $\delta$ is small enough.
\bigskip

\paragraph{\textmd{\textit{Statement (2)}}}

The argument is analogous to the case of $\PSu$. To begin with, using a suitable affine transformation we produce a function $\chat$ such that
\begin{equation}
d_{\PSd}(\chat,c_{0})\leq r/2,
\nonumber
\end{equation}
and such that $\chat$ does not saturate the inequalities in the definition of $\PSd$, namely there exists $\ep_{0}>0$ such that $\chat$ satisfies (\ref{non-sat-c}), (\ref{non-sat-c'}), and
\begin{equation}
|\chat''(t)|\leq
(1-\ep_{0})\frac{\omega(t)^{2}}{t^{2}}\exp\left(\psi(t)\right)
\qquad
\forall t\in(0,T_{0}].
\label{non-sat-c''}
\end{equation}

Once that $\chat$ has been chosen, we define $\cd$ as in (\ref{defn:cd}), and we claim that, when $\delta$ is small enough, it turns out that $d_{\PSd}(\cd,\chat)\leq r/2$ and $\cd\in\PSd$. The argument is the same as before, but now we have to verify also the condition on the second derivative, namely
\begin{equation}
|\cd''(t)|\leq\frac{\omega(t)^{2}}{t^{2}}\exp\left(\psi(t)\right)
\qquad
\forall t\in(0,T_{0}].
\label{est:cd''}
\end{equation}

Again this is true almost for free in the interval $[0,\delta]$, where the function is constant, and in the interval $[2\delta,T_{0}]$, where $\cd''(t)$ coincides with $\chat''(t)$, which in turn satisfies the stronger condition (\ref{non-sat-c''}). In the interval $[\delta,2\delta]$ the second derivative is given by
\begin{equation}
\cd''(t)=\theta''\left(\frac{t-\delta}{\delta}\right)\frac{1}{\delta^{2}}(\chat(t)-\chat(\delta))+
2\theta'\left(\frac{t-\delta}{\delta}\right)\frac{1}{\delta}\chat'(t)+
\theta\left(\frac{t-\delta}{\delta}\right)\chat''(t).
\nonumber
\end{equation}

Let us estimate the three terms separately. As for the first one, from the monotonicity of $\omega(t)$ we deduce that
\begin{eqnarray*}
\left|\theta''\left(\frac{t-\delta}{\delta}\right)\frac{1}{\delta^{2}}(\chat(t)-\chat(\delta))\right|
& \leq & \|\theta''\|_{\infty}\frac{|\chat(t)-\chat(\delta)|}{\delta^{2}}  \\[0.5ex]
& \leq & \frac{\|\theta''\|_{\infty}\Gamma_{\delta}}{\delta^{2}}\cdot\frac{t^{2}}{\omega(t)^{2}}
\cdot\frac{\omega(t)^{2}}{t^{2}}  \\[1ex]
& \leq & \frac{4\|\theta''\|_{\infty}\Gamma_{\delta}}{\omega(2\delta)^{2}}
\cdot\frac{\omega(t)^{2}}{t^{2}}\exp\left(\psi(t)\right).
\end{eqnarray*}

As for the second one, from (\ref{hp:fdl-c'}) and the monotonicity of $\omega(t)$ we deduce that
\begin{eqnarray*}
\left|2\theta'\left(\frac{t-\delta}{\delta}\right)\frac{1}{\delta}\chat'(t)\right|
& \leq & \frac{2\|\theta'\|_{\infty}}{\delta}|\chat'(t)| \\
& \leq & \frac{2\|\theta'\|_{\infty}}{\delta}\cdot\frac{\omega(t)}{t} \\[0.5ex]
& = & \frac{2\|\theta'\|_{\infty}}{\delta}\cdot
\frac{t}{\omega(t)}\cdot\frac{\omega(t)^{2}}{t^{2}} \\[0.5ex]
& \leq & \frac{4\|\theta'\|_{\infty}}{\omega(2\delta)}\cdot
\frac{\omega(t)^{2}}{t^{2}}\exp\left(\psi(t)\right).
\end{eqnarray*}

As for the third one, from (\ref{non-sat-c''}) we deduce that
\begin{equation}
\left|\theta\left(\frac{t-\delta}{\delta}\right)\chat''(t)\right|\leq
|\chat''(t)|\leq
(1-\ep_{0})\frac{\omega(t)^{2}}{t^{2}}\exp\left(\psi(t)\right).
\nonumber
\end{equation}

From all these estimates we conclude that
\begin{equation}
|\cd''(t)|\leq\left\{
\frac{4\|\theta''\|_{\infty}\Gamma_{\delta}}{\omega(2\delta)^{2}}+
\frac{4\|\theta'\|_{\infty}}{\omega(2\delta)}+
1-\ep_{0}\right\}
\frac{\omega(t)^{2}}{t^{2}}\exp\left(\psi(t)\right),
\nonumber
\end{equation}
which implies (\ref{est:cd''}) when $\delta$ is small enough because $\Gamma_{\delta}\to 0$ and $\omega(2\delta)\to +\infty$.
\end{proof}


\subsection{Asymptotic activators for initially constant functions}\label{sec:as-act-icf}

This subsection is the technical core of our construction. We show that, for every propagation speed $\chat(t)$ that is initially constant and belongs to the classes we are interested in, there exists a family $\{\cl(t)\}$ of asymptotic activators that converges to $\chat(t)$ in the same class. Since we already know that initially constant propagation speeds are dense, this is exactly what we need in order to apply Proposition~\ref{prop:as2un} and deduce the existence of universal activators, which in turn lead to the infinite derivative loss.

To begin with, we fix notations. Let $T_{0}$, $T_{1}$, $\mu_{1}$, $\mu_{2}$, $\gamma$ be positive real numbers such that
\begin{equation}
0<T_{1}<T_{0}
\qquad\mbox{and}\qquad
0<\mu_{1}<\gamma^{2}<\mu_{2}.
\nonumber
\end{equation}

Let $\chat:[0,T_{0}]\to[\mu_{1},\mu_{2}]$ be a function such that
\begin{equation}
\chat(t)=\gamma^{2}
\qquad
\forall t\in[0,T_{1}].
\nonumber
\end{equation}

For every large enough real number $\lambda$, let $\al$ and $\bl$ be real numbers such that
\begin{equation}
0<\al<2\al<\frac{\bl}{2}<\bl<T_{1},
\label{hp:al-bl-1}
\end{equation}
and 
\begin{equation}
\frac{\gamma\lambda\al}{2\pi}\in\n
\qquad\quad\mbox{and}\quad\qquad
\frac{\gamma\lambda\bl}{2\pi}\in\n.
\label{hp:al-bl-2}
\end{equation}

Let $\omega:(0,T_{0}]\to(0,+\infty)$ be a nonincreasing function satisfying (\ref{hp:omega2infty}), and let us set
\begin{equation}
\oml:=\min\{\omega(\bl),\log\lambda\}.
\label{defn:oml}
\end{equation}

Let us choose a cutoff function $\theta:\re\to\re$ of class $C^{\infty}$ satisfying (\ref{defn:cutoff}), and let us define $\epl:[0,T_{0}]\to\re$ by
\begin{equation}
\epl(t):=\left\{
\begin{array}{l@{\qquad}l}
0 & \mbox{if }t\in[0,\al]\cup[\bl,T_{0}], \\[1ex]
\dfrac{\oml}{t} & \mbox{if }t\in[2\al,\bl/2], \\[2.5ex]
\dfrac{\oml}{t}\cdot\theta\left(\dfrac{t-\al}{\al}\right) & \mbox{if }t\in[\al,2\al], \\[3ex]
\dfrac{\oml}{t}\cdot\theta\left(\dfrac{2(\bl-t)}{\bl}\right) & \mbox{if }t\in[\bl/2,\bl].
\end{array}\right.
\nonumber
\end{equation}

In words, $\epl(t)$ is a function of class $C^{\infty}$ that vanishes outside $[\al,\bl]$, and coincides with $\oml/t$ in the inner interval $[2\al,\bl/2]$.  Finally, let us define $\cl:[0,T_{0}]\to\re$ as
\begin{equation}
\cl(t):=\chat(t)-
\frac{\epl(t)}{4\gamma\lambda}\sin(2\gamma\lambda t)-
\frac{\epl'(t)}{8\gamma^{2}\lambda^{2}}\sin^{2}(\gamma\lambda t)-
\frac{\epl(t)^{2}}{64\gamma^{4}\lambda^{2}}\sin^{4}(\gamma\lambda t)
\label{defn:cl}
\end{equation}
for every $t\in[0,T_{0}]$. We observe that $\cl(t)$ coincides with $\chat(t)$ in the intervals $[0,\al]$ and $[\bl,T_{1}]$, where it is actually constantly equal to $\gamma^{2}$, and in the interval $[T_{1},T_{0}]$. In the interval $[\al,\bl]$ the function $\cl(t)$ is close to the constant $\gamma^{2}$, but it has a highly oscillatory behavior due to the trigonometric terms.


Now we check that $\{\cl(t)\}$ is a family of asymptotic activators converging to $\chat(t)$. To this end, in the next two results we show first the growth of solutions required in (\ref{eqn:asympt-activ}), and then the convergence $\cl(t)\to \chat(t)$ in the appropriate space.

\begin{prop}[Asymptotic activators -- Growth of solutions]\label{prop:growth}

Let $\chat(t)$, $\al$, $\bl$, $\omega(t)$, $\oml$, and $\cl(t)$ be as in the construction described above.

Let us assume that $\bl\to 0$ (and hence also $\al\to 0$ and $\oml\to +\infty$), and that
\begin{equation}
\lim_{\lambda\to +\infty}\frac{\oml}{\lambda\al}=0
\qquad\quad\mbox{and}\qquad\quad
\lim_{\lambda\to +\infty}\frac{\bl}{\al}=+\infty.
\label{hp:abl-1b}
\end{equation}

Then $\{\cl(t)\}$ is a family of asymptotic activators with rate
\begin{equation}
\phi(\lambda):=\frac{\oml}{32\gamma^{2}}\log\left(\frac{\bl}{\al}\right).
\label{defn:phi}
\end{equation}

\end{prop}

\begin{proof}

Let $\{\vl(t)\}$ denote the family of solutions to (\ref{eqn:uc-lambda})--(\ref{data:ode-lambda}). According to Definition~\ref{defn:un-act} we need to check that, for every $\delta\in(0,T_{0})$, there exist two positive constants $M_{\delta}$ and $\lambda_{\delta}$ such that
\begin{equation}
|u_{\lambda}'(t)|^{2}+\lambda^{2}|u_{\lambda}(t)|^{2}\geq
M_{\delta}\exp(2\phi(\lambda))
\qquad
\forall t\in[\delta,T_{0}]
\quad
\forall\lambda\geq\lambda_{\delta}.
\label{th:as-act}
\end{equation}

So let us assume that $\delta\in(0,T_{0})$ has been fixed. We recall that $\chat(t)$ is constant, equal to $\gamma^{2}$, in some interval $[0,T_{1}]$. Without loss of generality, we can always assume that $T_{1}<\delta$. 
Now we estimate $\vl(t)$ in the four subintervals $[0,\al]$, $[\al,\bl]$, $[\bl,T_{1}]$, $[T_{1},T_{0}]$, even if (\ref{th:as-act}) involves only values of $t$ greater than or equal to $\delta$.

\paragraph{\textmd{\textit{Estimate in $[0,\al]$}}}

In this interval $\cl(t)\equiv\gamma^{2}$, and therefore 
\begin{equation}
\vl(t)=\frac{1}{\gamma\lambda}\sin(\gamma\lambda t)
\qquad
\forall t\in[0,\al].
\nonumber
\end{equation}

Since $\gamma\lambda\al$ is an integer multiple of $2\pi$, we deduce that $\vl(\al)=0$ and $\vl'(\al)=1$.

\paragraph{\textmd{\textit{Estimate in $[\al,\bl]$}}}

In this interval $\cl(t)$ is given by (\ref{defn:cl}) with $c_{*}(t)\equiv\gamma^{2}$, and it possible to check that $\vl(t)$ has the explicit form
\begin{equation}
\vl(t)=\frac{1}{\gamma\lambda}\sin(\gamma\lambda t)
\exp\left(\frac{1}{8\gamma^{2}}\int_{\al}^{t}\epl(s)\sin^{2}(\gamma\lambda s)\,ds\right)
\qquad
\forall t\in[\al,\bl].
\nonumber
\end{equation}

Since $\gamma\lambda\bl$ is an integer multiple of $2\pi$, we deduce that $\vl(\bl)=0$ and
\begin{equation}
\vl'(\bl)=\exp\left(\frac{1}{8\gamma^{2}}\int_{\al}^{\bl}\epl(s)\sin^{2}(\gamma\lambda s)\,ds\right).
\label{eqn:ul'}
\end{equation}

Let us estimate the integral from below. Since $\epl(s)$ is always nonnegative, we can limit ourselves to the subinterval $[2\al,\bl/2]$, where $\epl(s)=\oml/s$. We obtain that
\begin{eqnarray*}
\int_{\al}^{\bl}\epl(s)\sin^{2}(\gamma\lambda s)\,ds & \geq &
\oml\int_{2\al}^{\bl/2}\frac{\sin^{2}(\gamma\lambda s)}{s}\,ds \\[1ex]
& = & 
\frac{\oml}{2}\int_{2\al}^{\bl/2}\left(\frac{1}{s}-\frac{\cos(2\gamma\lambda s)}{s}\right)\,ds \\[1ex]
& = & \frac{\oml}{2}\left\{\log\left(\frac{\bl}{4\al}\right)-\int_{2\al}^{\bl/2}
\frac{\cos(2\gamma\lambda s)}{s}\,ds\right\}.
\end{eqnarray*}

In order to estimate the last integral we integrate by parts. Since $2\gamma\lambda s$ is an integer multiple of $\pi$ when $s=2\al$ and $s=\bl/2$, we obtain that
\begin{equation}
\int_{2\al}^{\bl/2}\frac{\cos(2\gamma\lambda s)}{s}\,ds=
\frac{1}{2\gamma\lambda}\int_{2\al}^{\bl/2}\frac{\sin(2\gamma\lambda s)}{s^{2}}\,ds\leq
\frac{1}{2\gamma\lambda}\cdot\frac{1}{2\al},
\nonumber
\end{equation}
and hence
\begin{equation}
\int_{\al}^{\bl}\epl(s)\sin^{2}(\gamma\lambda s)\,ds\geq
\frac{\oml}{2}\left\{\log\left(\frac{\bl}{4\al}\right)-\frac{1}{4\gamma\lambda\al}\right\}.
\nonumber
\end{equation}

Due to the limits (\ref{hp:abl-1b}), when $\lambda$ is large enough we deduce that
\begin{equation}
\int_{\al}^{\bl}\epl(s)\sin^{2}(\gamma\lambda s)\,ds\geq
\frac{\oml}{4}\log\left(\frac{\bl}{\al}\right)=
8\gamma^{2}\phi(\lambda).
\label{est:int-b-phi}
\end{equation}

Plugging this estimate into (\ref{eqn:ul'}) we conclude that
\begin{equation}
|\vl'(\bl)|^{2}+\gamma^{2}\lambda^{2}|\vl(\bl)|^{2}\geq
\exp\left(2\phi(\lambda)\right).
\nonumber
\end{equation}

\paragraph{\textmd{\textit{Estimate in $[\bl,T_{1}]$}}}

In this interval $\cl(t)\equiv\gamma^{2}$, and therefore throughout this interval the quantity $|\vl'(t)|^{2}+\gamma^{2}\lambda^{2}|\vl(t)|^{2}$ remains constant, so that again
\begin{equation}
|\vl'(T_{1})|^{2}+\gamma^{2}\lambda^{2}|\vl(T_{1})|^{2}\geq
\exp\left(2\phi(\lambda)\right).
\label{est:vl(delta)}
\end{equation}

\paragraph{\textmd{\textit{Estimate in $[T_{1},T_{0}]$}}}

In this interval $\cl(t)$ coincides with $c_{*}(t)$. If we consider the hyperbolic energy
\begin{equation}
\eh(t):=|\vl'(t)|^{2}+\lambda^{2}c_{*}(t)|\vl(t)|^{2},
\nonumber
\end{equation}
its time-derivative satisfies
\begin{equation}
\eh'(t)=\lambda^{2}c_{*}'(t)|\vl(t)|^{2}\geq
-\frac{|c_{*}'(t)|}{c_{*}(t)}\eh(t),
\nonumber
\end{equation}
from which we deduce that
\begin{equation}
\eh(t)\geq
\eh(T_{1})\exp\left(-\int_{T_{1}}^{T_{0}}\frac{|c_{*}'(s)|}{c_{*}(s)}\,ds\right)
\qquad
\forall t\in[T_{1},T_{0}].
\nonumber
\end{equation}

Finally, we observe that $\eh(T_{1})$ coincides with the left-hand side of (\ref{est:vl(delta)}), and therefore we conclude that
\begin{eqnarray*}
|\vl'(t)|^{2}+\lambda^{2}|\vl(t)|^{2} & \geq &
\min\left\{1,\frac{1}{\mu_{2}}\right\}\eh(t) \\[1ex]
& \geq & 
\min\left\{1,\frac{1}{\mu_{2}}\right\}\eh(T_{1})
\exp\left(-\int_{T_{1}}^{T_{0}}\frac{|c_{*}'(s)|}{c_{*}(s)}\,ds\right)  \\[0.5ex]
& \geq &
\min\left\{1,\frac{1}{\mu_{2}}\right\}
\exp\left(-\int_{T_{1}}^{T_{0}}\frac{|c_{*}'(s)|}{c_{*}(s)}\,ds\right)\exp(2\phi(\lambda))
\end{eqnarray*}
for every $t\in[T_{1},T_{0}]$, and a fortiori for every $t\in[\delta,T_{0}]$. This proves (\ref{th:as-act}) with
\begin{equation}
M_{\delta}:=\min\left\{1,\frac{1}{\mu_{2}}\right\}
\exp\left(-\int_{T_{1}}^{T_{0}}\frac{|c_{*}'(s)|}{c_{*}(s)}\,ds\right),
\nonumber
\end{equation}
and $\lambda$ large enough so that (\ref{est:int-b-phi}) holds true. We observe that $M_{\delta}$ does depend on $\delta$ since at the beginning we might have reduced $T_{1}$ in order to have $T_{1}<\delta$.
\end{proof}


\begin{prop}[Asymptotic activators -- Convergence]\label{prop:convergence}

Let $\chat(t)$, $\al$, $\bl$, $\omega(t)$, $\oml$, and $\cl(t)$ be as in the construction described above.

Let us assume that $\bl\to 0$ (and hence also $\al\to 0$ and $\oml\to+\infty$), and that
\begin{equation}
\lim_{\lambda\to +\infty}\frac{\oml}{\lambda\al}=0.
\label{hp:abl-1}
\end{equation}

Then the following statements hold true.

\begin{enumerate}
\renewcommand{\labelenumi}{(\arabic{enumi})}

\item  If $\chat\in\PSu(T_{0},\mu_{1},\mu_{2},\omega)$, then it turns out that $\cl\in\PSu(T_{0},\mu_{1},\mu_{2},\omega)$ when $\lambda$ is large enough, and 
\begin{equation}
\lim_{\lambda\to +\infty}d_{\PSu}(\cl,\chat)=0.
\nonumber
\end{equation}

\item  Let $\psi:(0,T_{0}]\to(0,+\infty)$ be a nonincreasing function such that
\begin{equation}
\lim_{\lambda\to +\infty}\frac{\lambda\bl}{\omega(\bl)}\exp(-\psi(\bl))=
\lim_{\lambda\to +\infty}\frac{\oml}{\omega(\bl)}\exp(-\psi(\bl))=0.
\label{hp:abl-2}
\end{equation}

If $\chat\in\PSd(T_{0},\mu_{1},\mu_{2},\omega,\psi)$, then it turns out that $\cl\in\PSd(T_{0},\mu_{1},\mu_{2},\omega,\psi)$ when $\lambda$ is large enough, and 
\begin{equation}
\lim_{\lambda\to +\infty}d_{\PSd}(\cl,\chat)=0.
\nonumber
\end{equation}

\end{enumerate}

\end{prop}

\begin{proof}

For the first statement we have to check that, when $\lambda$ is large enough, $\cl$ is of class $C^{1}$ in $[0,T_{0}]$ (we can include the origin because by definition $\cl$ is initially constant) and satisfies the two inequalities
\begin{equation}
\mu_{1}\leq\cl(t)\leq\mu_{2}
\qquad
\forall t\in[0,T_{0}],
\label{th:cl-sh}
\end{equation}
and
\begin{equation}
|\cl'(t)|\leq\frac{\omega(t)}{t}
\qquad
\forall t\in(0,T_{0}],
\label{th:cl'}
\end{equation}
and moreover $\cl(t)$ converges to $\chat(t)$ in the sense that
\begin{equation}
\cl(t)\to \chat(t)
\quad
\mbox{uniformly in }[0,T_{0}],
\label{th:cl-unif}
\end{equation}
and
\begin{equation}
\forall\tau\in(0,T_{0})
\qquad
\cl'(t)\to \chat'(t)
\quad
\mbox{uniformly in }[\tau,T_{0}],
\label{th:cl'-unif}
\end{equation}

For the second statement we have to check the same properties, and in addition that $\cl$ is of class $C^{2}$ in $[0,T_{0}]$ and satisfies the estimate
\begin{equation}
|\cl''(t)|\leq\frac{\omega(t)^{2}}{t^{2}}\exp\left(\psi(t)\right)
\qquad
\forall t\in(0,T_{0}],
\label{th:cl''}
\end{equation}
and moreover
\begin{equation}
\forall\tau\in(0,T_{0})
\qquad
\cl''(t)\to \chat''(t)
\qquad
\mbox{uniformly in }[\tau,T_{0}],
\label{th:cl''-unif}
\end{equation}

As for the regularity, we observe that $\cl$ has the same $C^{1}$ or $C^{2}$ regularity of $\chat$, because all the terms we added in (\ref{defn:cl}) are of class $C^{\infty}$. 

As for the uniform convergence in (\ref{th:cl'-unif}) and (\ref{th:cl''-unif}), we observe that it is substantially trivial because, when $\tau$ is fixed, it turns out that $\cl(t)$ and $\chat(t)$ coincide in $[\tau,T_{0}]$ as soon as $\bl\leq\tau$, and this is true eventually because $\bl\to 0$.

\paragraph{\textmd{\textit{Estimates on $\epl(t)$}}}

We prove that
\begin{equation}
|\epl(t)|\leq\frac{\oml}{t}
\qquad
\forall t\in[\al,\bl],
\label{est:epl}
\end{equation}
and that for every positive integer $m$ there exists a constant $K_{m}$ such that the $m$-th derivative $\epl^{(m)}(t)$ satisfies (actually we need the estimate just for $m=1,2,3$)
\begin{equation}
|\epl^{(m)}(t)|\leq K_{m}\frac{\oml}{t^{m+1}}
\qquad
\forall t\in[\al,\bl].
\label{est:epl'}
\end{equation}

Indeed, estimate (\ref{est:epl}) follows from the definition of $\epl(t)$ because $\|\theta\|_{\infty}\leq 1$. As for (\ref{est:epl'}), we distinguish three cases. In the central interval $[2\al,\bl/2]$ we have the explicit expression $\epl(t)=\oml/t$, and therefore also derivatives can be explicitly computed, and they satisfy (\ref{est:epl'}). 

In the interval $[\al,2\al]$ we use Leibniz rule, and we obtain that
\begin{equation}
\epl^{(m)}(t)=
\oml\sum_{i=0}^{m}\binom{m}{i}\theta^{(i)}\left(\frac{t-\al}{\al}\right)
\frac{1}{\al^{i}}\cdot(-1)^{m-i}\frac{1}{t^{m-i+1}}(m-i)!
\nonumber
\end{equation}
for every $t\in[\al,2\al]$. Since $\al\geq t/2$ in this interval, we deduce that
\begin{equation}
\left|\epl^{(m)}(t)\right|\leq\oml
\sum_{i=0}^{m}\binom{m}{i}\|\theta^{(i)}\|_{\infty}
\frac{2^{i}}{t^{i}}\cdot\frac{1}{t^{m-i+1}}(m-i)!
\qquad
\forall t\in[\al,2\al],
\nonumber
\end{equation}
so that (\ref{est:epl'}) holds true in this interval with a suitable $K_{m}$. In the interval $[\bl/2,\bl]$ the proof is analogous.

\paragraph{\textmd{\textit{Estimates and convergence for $\cl(t)$}}}

We prove (\ref{th:cl-unif}), which in turn implies (\ref{th:cl-sh}) because $\cl(t)$ and $\chat(t)$ differ only in $[\al,\bl]$, and for $\lambda$ large this interval is contained in the initial interval $[0,T_{1}]$ where $\chat(t)$ is constant and strictly between $\mu_{1}$ and $\mu_{2}$.

In the interval $[\al,\bl]$ we exploit (\ref{est:epl}) and (\ref{est:epl'}) with $m=1$, and we deduce that
\begin{eqnarray*}
|\cl(t)-\chat(t)| & \leq &
\frac{1}{4\gamma\lambda}|\epl(t)|+
\frac{1}{8\gamma^{2}\lambda^{2}}|\epl'(t)|+
\frac{1}{64\gamma^{4}\lambda^{2}}|\epl(t)|^{2} \\[0.5ex]
& \leq & 
\frac{1}{4\gamma\lambda}\cdot\frac{\oml}{t}+
\frac{1}{8\gamma^{2}\lambda^{2}}\cdot\frac{K_{1}\oml}{t^{2}}+
\frac{1}{64\gamma^{4}\lambda^{2}}\cdot\frac{\oml^{2}}{t^{2}} \\[0.5ex]
& \leq & 
\frac{1}{4\gamma}\cdot\frac{\oml}{\lambda\al}+
\frac{K_{1}}{8\gamma^{2}}\cdot\frac{\oml}{\lambda^{2}\al^{2}}+
\frac{1}{64\gamma^{4}}\cdot\frac{\oml^{2}}{\lambda^{2}\al^{2}}. 
\end{eqnarray*}

From assumption (\ref{hp:abl-1}) it follows that all terms tend to~0 as $\lambda\to+\infty$, as required.

\paragraph{\textmd{\textit{Estimates on first derivatives}}}

We prove that estimate (\ref{th:cl'}) holds true when $\lambda$ is large enough. To this end, we can limit ourselves to the interval $[\al,\bl]$, because otherwise $\cl'(t)$ coincides with $\chat'(t)$, which satisfies the required estimate. Computing the time-derivative of $\cl(t)$, we discover that
\begin{equation}
\cl'(t)=-\sum_{i=1}^{5}L_{i}(t),
\nonumber
\end{equation}
where
\begin{gather*}
L_{1}(t):=\frac{1}{2}\epl(t)\cos(2\gamma\lambda t),
\qquad
L_{2}(t):=\frac{1}{16\gamma^{3}\lambda}\epl(t)^{2}\sin^{3}(\gamma\lambda t)\cos(\gamma\lambda t),  \\[1ex]
L_{3}(t):=\frac{3}{8\gamma\lambda}\epl'(t)\sin(2\gamma\lambda t),
\qquad
L_{4}(t):=\frac{1}{32\gamma^{4}\lambda^{2}}\epl(t)\epl'(t)\sin^{4}(\gamma\lambda t),  \\[1ex]
L_{5}(t):=\frac{1}{8\gamma^{2}\lambda^{2}}\epl''(t)\sin^{2}(\gamma\lambda t).
\end{gather*}

All these terms can be estimates by exploiting (\ref{est:epl}) and (\ref{est:epl'}) with $m=1,2$, and the fact that $t\geq\al$ and $\oml\leq\omega(t)$ for every $t\in[\al,\bl]$. We obtain that
\begin{gather*}
|L_{1}(t)|\leq
\frac{1}{2}\epl(t)\leq
\frac{1}{2}\frac{\oml}{t}\leq
\frac{1}{2}\frac{\omega(t)}{t},  \\[1ex]
|L_{2}(t)|\leq
\frac{1}{16\gamma^{3}\lambda}\epl(t)^{2}\leq
\frac{1}{16\gamma^{3}\lambda}\cdot\frac{\oml^{2}}{t^{2}}\leq
\frac{1}{16\gamma^{3}}\cdot\frac{\oml}{\lambda\al}\cdot\frac{\omega(t)}{t},  \\[1ex]
|L_{3}(t)|\leq
\frac{3}{8\gamma\lambda}|\epl'(t)|\leq
\frac{3}{8\gamma\lambda}\cdot\frac{K_{1}\oml}{t^{2}}\leq
\frac{3K_{1}}{8\gamma}\cdot\frac{1}{\lambda\al}\cdot\frac{\omega(t)}{t},  \\[1ex]
|L_{4}(t)|\leq
\frac{1}{32\gamma^{4}\lambda^{2}}|\epl(t)|\cdot|\epl'(t)|\leq
\frac{1}{32\gamma^{4}\lambda^{2}}\cdot\frac{K_{1}\oml^{2}}{t^{3}}\leq
\frac{K_{1}}{32\gamma^{4}}\cdot\frac{\oml}{\lambda^{2}\al^{2}}\cdot\frac{\omega(t)}{t},  \\[1ex]
|L_{5}(t)|\leq
\frac{1}{8\gamma^{2}\lambda^{2}}|\epl''(t)|\leq
\frac{1}{8\gamma^{2}\lambda^{2}}\cdot\frac{K_{2}\oml}{t^{3}}\leq
\frac{K_{2}}{8\gamma^{2}}\cdot\frac{1}{\lambda^{2}\al^{2}}\cdot\frac{\omega(t)}{t}.
\end{gather*}

From all these estimates we conclude that
\begin{equation}
|\cl'(t)|\leq
\left(\frac{1}{2}+G_{1}(\lambda)\right)\frac{\omega(t)}{t}
\qquad
\forall t\in[\al,\bl],
\nonumber
\end{equation}
where $G_{1}(\lambda)\to 0$ as $\lambda\to +\infty$ because of (\ref{hp:abl-1}). This implies that (\ref{th:cl'}) holds true when $\lambda$ is large enough.

\paragraph{\textmd{\textit{Estimates on second derivatives}}}

We prove that (\ref{th:cl''}) holds true when $\lambda$ is large enough. As in the case of first derivatives, we can limit ourselves to the interval $[\al,\bl]$, because otherwise $\cl''(t)$ coincides with $\chat''(t)$, which satisfies the required estimate. Computing the second order time-derivative of $\cl(t)$, we discover that
\begin{equation}
\cl''(t)=-\sum_{i=1}^{8}L_{i}(t),
\nonumber
\end{equation}
where now
\begin{gather*}
L_{1}(t):=-\gamma\lambda\epl(t)\sin(2\gamma\lambda t),
\qquad
L_{2}(t):=\frac{\epl(t)^{2}}{16\gamma^{2}}\left(-\sin^{4}(\gamma\lambda t)+3\sin^{2}(\gamma\lambda t)\cos^{2}(\gamma\lambda t)\right),  \\[1ex]
L_{3}(t):=\frac{5}{4}\epl'(t)\cos(2\gamma\lambda t),
\qquad
L_{4}(t):=\frac{1}{32\gamma^{4}\lambda^{2}}\epl'(t)^{2}\sin^{4}(\gamma\lambda t),  \\[1ex]
L_{5}(t):=\frac{1}{4\gamma^{3}\lambda}\epl(t)\epl'(t)\sin^{3}(\gamma\lambda t)\cos(\gamma\lambda t),
\qquad
L_{6}(t):=\frac{1}{2\gamma\lambda}\epl''(t)\sin(2\gamma\lambda t),  \\[1ex]
L_{7}(t):=\frac{1}{32\gamma^{4}\lambda^{2}}\epl(t)\epl''(t)\sin^{4}(\gamma\lambda t),
\qquad
L_{8}(t):=\frac{1}{8\gamma^{2}\lambda^{2}}\epl'''(t)\sin^{2}(\gamma\lambda t).
\end{gather*}

As we did in the case of first derivatives, all these terms can be estimates by exploiting (\ref{est:epl}) and (\ref{est:epl'}) with $m=1,2,3$, and the fact that $t\geq\al$ and $\oml\leq\omega(t)$ for every $t\in[\al,\bl]$. In particular, for the terms from $L_{3}(t)$ to $L_{8}(t)$ we obtain that
\begin{equation}
\sum_{i=3}^{8}|L_{i}(t)|\leq G_{2}(\lambda)\frac{\omega(t)^{2}}{t^{2}},
\nonumber
\end{equation}
where $G_{2}(\lambda)\to 0$ as $\lambda\to +\infty$ because of assumption (\ref{hp:abl-1}) and the fact that $\oml\to +\infty$. The first two terms require more delicate estimates, at least in the case where $\psi(t)$ is bounded. We estimate them as
\begin{equation}
|L_{2}(t)|\leq
\frac{1}{4\gamma^{2}}\epl(t)^{2}\leq
\frac{1}{4\gamma^{2}}\cdot\frac{\oml^{2}}{t^{2}}=
\frac{1}{4\gamma^{2}}\cdot\frac{\oml}{\omega(\bl)}\cdot\frac{\oml\omega(\bl)}{t^{2}}\leq
\frac{1}{4\gamma^{2}}\cdot\frac{\oml}{\omega(\bl)}\cdot\frac{\omega(t)^{2}}{t^{2}},
\nonumber
\end{equation}
and
\begin{equation}
|L_{1}(t)|\leq
\gamma\lambda\epl(t)\leq
\gamma\lambda\cdot\frac{\oml}{t}=
\gamma\lambda\frac{\bl}{\omega(\bl)}\cdot\frac{\oml\omega(\bl)}{\bl t}\leq
\gamma\frac{\lambda\bl}{\omega(\bl)}\cdot\frac{\omega(t)^{2}}{t^{2}}.
\nonumber
\end{equation}

From all these estimates we conclude that
\begin{equation}
|\cl''(t)|\leq
G_{3}(\lambda)\frac{\omega(t)^{2}}{t^{2}}\exp\left(\psi(\bl)\right)\leq
G_{3}(\lambda)\frac{\omega(t)^{2}}{t^{2}}\exp\left(\psi(t)\right),
\nonumber
\end{equation}
where
\begin{equation}
G_{3}(\lambda):=\left(G_{2}(\lambda)+
\frac{1}{4\gamma^{2}}\cdot\frac{\oml}{\omega(\bl)}+
\gamma\frac{\lambda\bl}{\omega(\bl)}
\right)\exp\left(-\psi(\bl)\right).
\nonumber
\end{equation}

Due to (\ref{hp:abl-2}) we conclude that $G_{3}(\lambda)\to 0$ as $\lambda\to +\infty$, and therefore (\ref{th:cl''}) is satisfied when $\lambda$ is large enough.
\end{proof}


\subsection{Choice of the parameters}\label{sec:parameters}

The proof of Theorem~\ref{thm:counter-c'} has been reduced to finding $\al\to 0$ and $\bl\to 0$ satisfying (\ref{hp:al-bl-1}) and (\ref{hp:al-bl-2}) for $\lambda$ large enough, and the conditions
\begin{equation}
\lim_{\lambda\to +\infty}\frac{\oml}{\lambda\al}=0,
\qquad
\lim_{\lambda\to +\infty}\frac{\bl}{\al}=+\infty,
\qquad
\lim_{\lambda\to +\infty}\frac{\oml}{\log\lambda}\log\left(\frac{\bl}{\al}\right)=+\infty,
\label{hp:al-bl-lim}
\end{equation}
where $\oml$ is defined by (\ref{defn:oml}). In the case of Theorem~\ref{thm:counter-c''} we need also the further conditions coming from (\ref{hp:abl-2}), namely 
\begin{equation}
\lim_{\lambda\to +\infty}\frac{\lambda\bl\exp(-\psi(\bl))}{\omega(\bl)}=0,
\qquad\qquad
\lim_{\lambda\to +\infty}\frac{\oml\exp(-\psi(\bl))}{\omega(\bl)}=0.
\label{hp:al-bl-lim-c''}
\end{equation}

Indeed, from Proposition~\ref{prop:growth} we know that the limits (\ref{hp:al-bl-lim}) guarantee that the family $\{\cl\}$ defined by (\ref{defn:cl}) is a family of asymptotic activators with a rate $\phi(\lambda)$ satisfying (\ref{hp:phi-log}). Moreover, from Proposition~\ref{prop:convergence} we know that the first limit in (\ref{hp:al-bl-lim}) guarantees the convergence of $\cl$ to $\chat$ in $\PSu$, while (\ref{hp:al-bl-lim-c''}) guarantees the convergence also in $\PSd$. Since $\chat$ is a generic element of a dense subset (Proposition~\ref{prop:icf}), this is enough to show that for every unbounded sequence $\{\lambda_{n}\}$ there exist a residual set of universal activators with a rate satisfying (\ref{hp:phi-log}). Finally, from Proposition~\ref{prop:ua2idl} we know that these universal activators cause an infinite derivative loss.

\paragraph{\textmd{\textit{Parameters for Theorem~\ref{thm:counter-c'}}}}

In this case we set 
\begin{equation}
\al:=\frac{2\pi}{\gamma\lambda}\left\lfloor\lambda^{1/4}\right\rfloor,
\qquad\qquad
\bl:=\frac{2\pi}{\gamma\lambda}\left\lfloor\lambda^{1/2}\right\rfloor,
\nonumber
\end{equation}
where $\lfloor\alpha\rfloor$ stands for the integer part of a real number $\alpha$.

Properties (\ref{hp:al-bl-1}) and (\ref{hp:al-bl-2}) are almost evident from the definition. For the first limit in (\ref{hp:al-bl-lim}), it is enough to recall that $\oml\leq\log\lambda$ and observe that $\lambda\al\sim\lambda^{1/4}$, up to multiplicative constants. For the second limit, we observe that $\bl/\al\sim\lambda^{1/4}$, and this proves also the third limit because $\log(\bl/\al)\sim\log\lambda$, and $\oml\to +\infty$ as $\lambda\to +\infty$.

\paragraph{\textmd{\textit{Parameters for Theorem~\ref{thm:counter-c''}}}}

In this case the choice of parameters is more subtle. To begin with, we observe that the limits in assumption (\ref{hp:counter-c''}) imply that
\begin{equation}
\oml\to +\infty
\qquad\quad\mbox{and}\quad\qquad
\Gamma_{\lambda}:=\frac{\omega(\lambda^{-1/2})\psi(\lambda^{-1/2})}{\log\lambda}\to +\infty.
\nonumber
\end{equation}

Now we set
\begin{equation}
\al:=\frac{2\pi}{\gamma\lambda}\left\lfloor\log\lambda\exp(\psil)\right\rfloor,
\qquad\qquad
\bl:=\frac{2\pi}{\gamma\lambda}\left\lfloor\log\lambda\exp(2\psil)\right\rfloor,
\nonumber
\end{equation}
where
\begin{equation}
\psil:=\min\left\{\frac{1}{8}\log\lambda,
\frac{1}{4}\psi\left(\frac{1}{\sqrt{\lambda}}\right)+\frac{1}{4}\log\Gamma_{\lambda}\right\}.
\label{defn:psil}
\end{equation}

We observe that $\psil\to +\infty$ and $\exp(2\psil)\leq\lambda^{1/4}$, and therefore
\begin{equation}
\bl\leq\frac{1}{\sqrt{\lambda}},
\label{est:bl-sqrt}
\end{equation}
when $\lambda$ is large enough. This proves in particular that $\al\to 0$ and $\bl\to 0$. Also properties (\ref{hp:al-bl-1}) and (\ref{hp:al-bl-2}) follow almost immediately from definition (\ref{defn:psil}) when $\lambda$ is large enough. Let us check the limits in (\ref{hp:al-bl-lim}) and (\ref{hp:al-bl-lim-c''}).

\begin{itemize}

\item For the first limit in (\ref{hp:al-bl-lim}) we observe that, up to multiplicative constants,
\begin{equation}
\frac{\oml}{\lambda\al}\sim
\frac{\oml}{\log\lambda}\exp(-\psil),
\nonumber
\end{equation}
and we conclude because $\oml\leq\log\lambda$ and $\psil\to +\infty$.

\item  The second limit in (\ref{hp:al-bl-lim}) is immediate because $\psil\to +\infty$.

\item  For the third limit in (\ref{hp:al-bl-lim}) we observe that
\begin{equation}
\frac{\oml}{\log\lambda}\log\left(\frac{\bl}{\al}\right)\sim
\frac{\oml\psil}{\log\lambda},
\nonumber
\end{equation}
and we claim that
\begin{equation}
\frac{\oml\psil}{\log\lambda}\geq
\min\left\{\frac{\oml}{8},\psil,\frac{\Gamma_{\lambda}}{4}\right\},
\label{claim:3-min}
\end{equation}
which implies the required limit because the three terms in the minimum tend to $+\infty$. In order to prove the claim, we distinguish three cases. If $\oml=\log\lambda$ or $\psil=\frac{1}{8}\log\lambda$, then (\ref{claim:3-min}) is almost immediate. It remains the case where
\begin{equation}
\oml=\omega(\bl)\geq\omega\left(\frac{1}{\sqrt{\lambda}}\right)
\nonumber
\end{equation}
where the inequality follows from (\ref{est:bl-sqrt}) and the monotonicity of $\omega$, and
\begin{equation}
\psil=\frac{1}{4}\psi\left(\frac{1}{\sqrt{\lambda}}\right)+\frac{1}{4}\log\Gamma_{\lambda}\geq
\frac{1}{4}\psi\left(\frac{1}{\sqrt{\lambda}}\right),
\nonumber
\end{equation}
in which case $\oml\psil/\log\lambda\geq\Gamma_{\lambda}/4$.

\item  In the first limit of (\ref{hp:al-bl-lim-c''}) we exploit again (\ref{est:bl-sqrt}), and thus from the monotonicity of $\omega$ and $\psi$ we deduce that
\begin{eqnarray*}
\frac{\lambda\bl}{\omega(\bl)}\exp(-\psi(\bl)) & \leq &
\frac{\lambda\bl}{\omega\left(\lambda^{-1/2}\right)}\exp(-\psi\left(\lambda^{-1/2})\right) \\[0.5ex]
& \leq & 
\frac{2\pi}{\gamma}\frac{\log\lambda}{\omega\left(\lambda^{-1/2}\right)}
\exp\left(2\psil\right)\exp\left(-\psi(\lambda^{-1/2})\right).
\end{eqnarray*}

On the other hand
\begin{eqnarray*}
\frac{\log\lambda}{\omega\left(\lambda^{-1/2}\right)}
\exp\left(2\psil-\psi(\lambda^{-1/2})\right) & \leq &
\frac{\log\lambda}{\omega\left(\lambda^{-1/2}\right)}
\exp\left(\frac{1}{2}\log\Gamma_{\lambda}-\frac{1}{2}\psi(\lambda^{-1/2})\right) \\[0.5ex]
& = & 
\frac{1}{\Gamma_{\lambda}^{1/2}}\cdot
\psi(\lambda^{-1/2})\exp\left(-\frac{1}{2}\psi(\lambda^{-1/2})\right),
\end{eqnarray*}
and the latter tends to 0 because $\Gamma_{\lambda}\to +\infty$ and the function $t\to t\exp(-t/2)$ is bounded for positive $t$.

\item  In the second limit of (\ref{hp:al-bl-lim-c''}) with an analogous argument we obtain that
\begin{equation}
\frac{\oml\exp(-\psi(\bl))}{\omega(\bl)}=
\frac{\oml}{\omega(\bl)\psi(\bl)}\cdot\psi(\bl)\exp(-\psi(\bl))\leq
\frac{1}{\Gamma_{\lambda}}\cdot\psi(\bl)\exp(-\psi(\bl)),
\nonumber
\end{equation} 
and again we conclude as above.

\end{itemize} 


\subsection{An ``explicit'' counterexample}\label{sec:iteration}

In this final subsection we show how one could combine our tools in order to provide a more ``explicit'' counterexample through an iterative procedure in the same spirit of the original construction introduced in~\cite{dgcs}. We limit ourselves to sketching the procedure in the case of Theorem~\ref{thm:counter-c'}, but the other case is analogous.

Let us start with the constant propagation speed $\gamma_{0}(t)\equiv(\mu_{1}+\mu_{2})/2$, and let us consider the family $\{\cl(t)\}$ defined by (\ref{defn:cl}) starting with $\chat(t):=\gamma_{0}(t)$. From Proposition~\ref{prop:growth} and Proposition~\ref{prop:convergence} we know that there exists a large enough $\lambda_{1}$ such that $c_{\lambda_{1}}(t)$ satisfies (\ref{hp:c-sh}) and (\ref{hp:fdl-c'}) with \emph{strict inequalities}, and the corresponding solution $u_{\lambda_{1}}(t)$ to (\ref{eqn:uc-lambda})--(\ref{data:ode-lambda}) satisfies
\begin{equation}
|u_{\lambda_{1}}'(t)|^{2}+\lambda_{1}^{2}|u_{\lambda_{1}}(t)|^{2}>
\exp(\phi(\lambda_{1}))
\qquad
\forall t\geq\frac{T_{0}}{2},
\label{activ-1}
\end{equation}
where $\phi$ is defined by (\ref{defn:phi}). We stress that again the \emph{inequality is strict}, and by compactness the ``strictness'' (here and before) is uniform in $t$. At this point we set $\gamma_{1}(t):=c_{\lambda_{1}}(t)$, and we interpret (\ref{activ-1}) by saying that $\gamma_{1}(t)$ activates the frequency $\lambda_{1}$ for $t\geq T_{0}/2$.

Now we consider a new family $\{\cl(t)\}$ defined again by (\ref{defn:cl}), but this time starting with $\chat(t):=\gamma_{1}(t)$. This is possible because $\gamma_{1}(t)$ is initially constant. Again we can find a large enough $\lambda_{2}$ such that $c_{\lambda_{2}}(t)$ satisfies (\ref{hp:c-sh}) and (\ref{hp:fdl-c'}) with strict inequalities, and the corresponding solution $u_{\lambda_{2}}(t)$ to (\ref{eqn:uc-lambda})--(\ref{data:ode-lambda}) satisfies
\begin{equation}
|u_{\lambda_{2}}'(t)|^{2}+\lambda_{2}^{2}|u_{\lambda_{2}}(t)|^{2}>
\exp(\phi(\lambda_{2}))
\qquad
\forall t\geq\frac{T_{0}}{4}.
\label{activ-2}
\end{equation}

At this point we set $\gamma_{2}(t):=c_{\lambda_{2}}(t)$, and we interpret (\ref{activ-2}) by saying that $\gamma_{2}(t)$ activates the frequency $\lambda_{2}$ for $t\geq T_{0}/4$. On the other hand, $\gamma_{2}(t)$ is as close as we need to $\gamma_{1}(t)$ in the uniform norm, and hence also the solution to
\begin{equation}
u_{\lambda_{1}}''(t)+\lambda_{1}^{2}\gamma_{2}(t)u_{\lambda_{1}}(t)=0
\nonumber
\end{equation}
(note that here we have $\lambda_{1}$ and $\gamma_{2}(t)$ in the same equation) with the usual initial data satisfies again the strict inequality (\ref{activ-1}), even if possibly in a ``less strict'' way. In other words, the new propagation speed $\gamma_{2}(t)$ activates in the same time the frequency $\lambda_{1}$ for $t\geq T_{0}/2$, and the frequency $\lambda_{2}$ for $t\geq T_{0}/4$.

This procedure can be iterated. More precisely, at step~$n$ we obtain a propagation speed $\gamma_{n}(t)$ and a real number $\lambda_{n}$ such that
\begin{itemize}

\item $\lambda_{n}\geq n$ (this guarantees that $\lambda_{n}\to +\infty$),

\item  $\gamma_{n}(t)$ is initially constant,

\item  $\gamma_{n}(t)=\gamma_{n-1}(t)$ for every $t\geq 2^{-n}$, and $|\gamma_{n}(t)-\gamma_{n-1}(t)|\leq 2^{-n}$ for every $t\in[0,2^{-n}]$,

\item  $\gamma_{n}(t)$ satisfies (\ref{hp:c-sh}) and (\ref{hp:fdl-c'}) with strict inequalities,

\item  solutions to (\ref{eqn:u-lambda}) with $c(t):=\gamma_{n}(t)$ and initial data (\ref{data:ode-lambda}) satisfy
\begin{equation}
|u_{\lambda_{i}}'(t)|^{2}+\lambda_{i}^{2}|u_{\lambda_{i}}(t)|^{2}>
\exp(\phi(\lambda_{i}))
\qquad
\forall t\geq\frac{T_{0}}{2^{i}}
\nonumber
\end{equation}
for every $i=1,\ldots,n$.

\end{itemize}

At this point it is possible to conclude that $\gamma_{n}(t)$ converges uniformly to some propagation speed $\gamma_{\infty}(t)$, which turns out to be a universal activator for the sequence $\{\lambda_{n}\}$.


\subsubsection*{\centering Acknowledgments}

Both authors are members of the \selectlanguage{italian} ``Gruppo Nazionale per l'Analisi Matematica, la Probabilit\`{a} e le loro Applicazioni'' (GNAMPA) of the ``Istituto Nazionale di Alta Matematica'' (INdAM). 

\selectlanguage{english}



\label{NumeroPagine}

\end{document}